\documentclass{amsart}

\usepackage{latexsym, amsfonts, amsmath, amssymb, amscd, graphicx}

\newtheorem{theorem}{Theorem}[section]

\newtheorem{lemma}[theorem]{Lemma}
\newtheorem{corollary}[theorem]{Corollary}
\newtheorem{prop}[theorem]{Proposition}

\theoremstyle{definition}

\newtheorem{remark}[theorem]{Remark}
\newtheorem{df}[theorem]{Definition}

\newcommand{\bg}{{\overline{g}}}
\newcommand{\be}{{\overline{e}}}
\newcommand{\bG}{{\overline{G}}}

\newcommand{\ul}[1]{\underline{#1}}
\newcommand{\wh}[1]{\widehat{#1}}
\newcommand{\wt}[1]{\widetilde{#1}}
\newcommand{\wb}[1]{\overline{#1}}

\newcommand{\veps}{\varepsilon}
\newcommand{\ve}{\varepsilon}

\newcommand{\ld}{\ldots}
\newcommand{\lspan}[1]{\mathrm{Span}\,\{#1\}}

\newcommand{\ot}{\otimes}

\newcommand{\ZZ}{\mathbb{Z}}

\newcommand{\FF}{\mathbb{F}}

\newcommand{\chr}[1]{\mathrm{char}\,#1}

\newcommand{\ad}[1]{\mathrm{ad}\,#1}
\newcommand{\Ad}[1]{\mathrm{Ad}\,#1}

\newcommand{\GL}{\mathrm{GL}}

\newcommand{\Lie}{\mathrm{Lie}}

\newcommand{\ca}{\mathcal{O}}
\newcommand{\M}{\mathfrak{M}}
\newcommand{\Wmn}{W(m;\ul{n})}
\newcommand{\Wm}{W(m;\ul{1})}
\newcommand{\Smn}{S(m;\ul{n})}
\newcommand{\Sm}{S(m;\ul{1})}
\newcommand{\Hmn}{H(m;\ul{n})}
\newcommand{\Hm}{H(m;\ul{1})}
\newcommand{\D}{\partial}
\newcommand{\pr}{\mathrm{pr}_1}
\newcommand{\oS}{\omega_S}
\newcommand{\oH}{\omega_H}

\newcommand{\dij}[3]{D_{#1,#2}(z^{#3})}
\newcommand{\DA}[1]{D_{1,2}(z^{#1})}
\newcommand{\DOT}[1]{D_{1,2}(z_1^{#1}z_2)}

\newcommand{\dr}[1]{D_H(z_1^{#1}z_{r+1})}
\newcommand{\da}[1]{D_H(z^{#1})}
\newcommand{\zr}{z_{r+1}}

\newcommand{\Hom}{\mathrm{Hom}\,}
\newcommand{\End}{\mathrm{End}\,}

\newcommand{\Aut}{\mathrm{Aut}}
\newcommand{\AAut}{\mathbf{Aut}}
\newcommand{\Der}{\mathrm{Der}}

\newcommand{\sd}{P}
\newcommand{\te}{\gamma}
\newcommand{\ke}{\kappa}
\newcommand{\grca}{\Gamma_\ca(G,\sd,\te)}
\newcommand{\grcal}{\Gamma_\ca(G,\sd,\ke,\te)}
\newcommand{\grW}{\Gamma_W(G,\sd,\te)}
\newcommand{\grS}{\Gamma_S(G,\sd,\te,g_0)}

\newcommand{\Gs}{\mathbf{G}}
\newcommand{\Hs}{\mathbf{H}}
\newcommand{\Q}{\mathbf{Q}}

\newcommand{\Gd}{\mathfrak{G}}
\newcommand{\Hd}{\mathfrak{H}}
\newcommand{\Prim}{\mathrm{Prim}}
\newcommand{\V}{\mathfrak{V}}
\newcommand{\sdp}[1]{{}^{#1}h}
\newcommand{\sdpo}[1]{{}^{#1}\delta}
\newcommand{\DP}{{}^p\delta}

\begin{document}

\title[Gradings on restricted Cartan type Lie algebras]{Group gradings on restricted Cartan type\\ Lie algebras}

\author[Bahturin]{Yuri Bahturin}
\address{Department of Mathematics and Statistics\\Memorial
University of Newfoundland\\ St. John's, NL, A1C5S7, Canada}
\email{bahturin@mun.ca}

\author[Kochetov]{Mikhail Kochetov}
\address{Department of Mathematics and Statistics\\Memorial
University of Newfoundland\\ St. John's, NL, A1C5S7, Canada}
\email{mikhail@mun.ca}

\thanks{The first author acknowledges support by by NSERC grant \# 227060-04. The second author acknowledges support by NSERC Discovery Grant \# 341792-07.}

\begin{abstract}
For a given abelian group $G$, we classify the isomorphism classes of $G$-gradings on the simple restricted Lie algebras of types $\Wm$ and $\Sm$ ($m\ge 2$), 
in terms of numerical and group-theoretical invariants. Our main tool is automorphism group schemes, which we determine for the simple restricted Lie algebras 
of types $\Sm$ and $\Hm$. The ground field is assumed to be algebraically closed of characteristic $p>3$. 
\end{abstract}

\maketitle


\section{Introduction}\label{intro}

Let $U$ be an algebra (not necessarily associative) over a field $\FF$ and let $G$ be a group, written multiplicatively.

\begin{df}\label{group_grad}
A {\em $G$-grading} on $U$ is a vector space decomposition
\[
U=\bigoplus_{g\in G} U_g
\]
such that 
\[
U_g U_h\subset U_{gh}\;\mbox{ for all }\;g,h\in G.
\]
$U_g$ is called the {\em homogeneous component} of degree $g$. The {\em support} of the $G$-grading is the set
\[
\{g\in G\;|\;U_g\neq 0\}.
\] 
\end{df}

If $U$ is finite-dimensional, then, replacing $G$ with the subgroup generated by the support of the grading, 
we may assume without loss of generality that $G$ is finitely generated.

\begin{df}\label{iso_grad}
We say that two $G$-gradings, $U=\bigoplus_{g\in G} U_g$ and $U=\bigoplus_{g\in G} U'_g$, are {\em isomorphic} if there exists an algebra automorphism $\psi:U\to U$ such that
\[
\psi(U_g)=U'_{g}\;\mbox{ for all }\;g\in G,
\]
i.e., $U=\bigoplus_{g\in G} U_g$ and $U=\bigoplus_{g\in G} U'_g$ are isomorphic as $G$-graded algebras.
\end{df}

Since the same vector space decomposition can be regarded as a $G$-grading for different groups $G$, one can also define {\em equivalence} of gradings, 
which is a weaker relation than isomorphism --- see e.g. \cite{Ksur} for a discussion. {\em Fine gradings} (i.e., those that cannot be refined) as well as their 
universal groups are of particular interest.    

We are interested in the problem of finding all possible group gradings on finite-dimensional simple Lie algebras.
If $L$ is a simple Lie algebra, then it is known that the support of any $G$-grading on $L$ generates an abelian group.
Hence, in this paper we will always assume that $G$ is {\em abelian}. We will also assume that the ground field $\FF$ is {\em algebraically closed}.

For an arbitrary abelian group $G$, the classification of $G$-gradings (up to isomorphism) is known for almost all classical simple Lie algebras over an algebraically 
closed field of characteristic $0$ or $p>2$. The classification of fine gradings (up to equivalence) is also known. 
The reader may refer to \cite{BK09,Eld09} and references therein.

Our goal is to carry out the same classification for simple restricted Lie algebras of Cartan type. In the present paper, we achieve this goal for Witt algebras 
and for special algebras (see definitions in Section \ref{section_definitions}) in characterisitc $p>3$: 
Theorem \ref{classification_gradings_W} with Corollary \ref{fine_gradings_W} and 
Theorem \ref{classification_gradings_S} with Corollary \ref{fine_gradings_S}, respectively. 

In a number of cases, a fruitful approach to the classification of gradings by abelian groups on an algebra $U$ is to use another algebra, $R$, 
that shares with $U$ the automorphism group scheme (see Section \ref{section_schemes}) 
and whose gradings are easier to study. This approach often requires equipping $R$ with some additional structure. 
For classical simple Lie algebras of series $A$, $B$, $C$ and $D$, $R$ is the matrix algebra $M_n(\FF)$, possibly equipped with an antiautomorphism.
For Lie algebras of Cartan type, $R$ is the ``coordinate algebra'' $\ca(m;\ul{n})$ (see Definition \ref{df_div_powers}), which, in the restricted case, 
is just the truncated polynomial algebra $\ca(m;\ul{1})=\FF[x_1,\ld,x_m]/(x_1^p,\ld,x_m^p)$. It is not difficult to classify $G$-gradings on the algebra $\ca(m;\ul{1})$ ---  
see Theorem \ref{classification_gradings_O} and Corollary \ref{fine_gradings_O}. Since it is known that $\ca(m;\ul{1})$ has the same automorphism group scheme 
as the Witt algebra $\Wm$, this immediately gives the classification of gradings for the latter. In Section \ref{section_schemes}, we show that the automorphism group 
scheme of the special algebra $\Sm^{(1)}$, $m\ge 3$ --- respectively, Hamiltonian algebra $\Hm^{(2)}$, $m=2r$ --- 
is isomorphic to the stabilizer of the differential form $\oS$ --- respectively, $\oH$ --- in the automorphism group scheme of $\ca(m;\ul{1})$; see Theorems 
\ref{iso_S} and \ref{iso_H}, respectively. 
As a consequence, all gradings on $\Sm^{(1)}$ and $\Hm^{(2)}$ come from gradings on $\ca(m;\ul{1})$. 
However, this does not yet give a classification of gradings on $\Sm^{(1)}$ and $\Hm^{(2)}$ up to isomorphism, because one must make sure that the 
isomorphism preserves the appropriate differential form --- see Corollary \ref{gradings_come_from_O}. With some more work, we obtain a classification of gradings on  
$\Sm^{(1)}$, $m\ge 3$, and $S(2;\ul{1})^{(2)}=H(2;\ul{1})^{(2)}$. The classification for $H(2r;\ul{1})^{(2)}$ with $r>1$ remains open.

The paper is structured as follows. In Section \ref{section_definitions}, we recall the definitions and basic facts regarding Lie algebras of Cartan type.
In Section \ref{section_schemes}, we briefly recall background information on automorphism group schemes and determine them for $\Sm^{(1)}$ and $\Hm^{(2)}$.
In Section \ref{section_gradings}, we obtain the classification of gradings for Witt and special algebras (using the results on automorphism group schemes).


\section{Cartan type Lie algebras}\label{section_definitions} 

We start by briefly recalling the definitions and relevant properties of Cartan type Lie algebras. We will use \cite{St} as a standard reference. 
Fix $m\geq 1$ and $\ul{n}=(n_1,\ld,n_m)$ where $n_i\geq 1$. Set 
\[
\ZZ^{(m;\ul{n})}:=\{\alpha\in\ZZ^m\;|\;0\leq\alpha_i<p^{n_i}\mbox{ for }i=1,\ldots,m\}.
\]
The elements of $\ZZ^{(m;\ul{n})}$ will be called multi-indices and denoted by Greek letters $\alpha$, $\beta$, $\gamma$. 
For $\alpha=(\alpha_1,\ld,\alpha_m)$, we set 
\[
|\alpha|=\alpha_1+\cdots+\alpha_m.
\]
We will denote by $\ul{1}$ the multi-index that has 1 in all positions and by $\veps_i$ the multi-index that has $1$ in position $i$ and zeros elsewhere. 

Let $\FF$ be a field of characteristic $p>0$. 

\begin{df}
\label{df_div_powers}
The algebra $\ca=\ca(m;\ul{n})$ over $\FF$ is a commutative associative algebra with a basis 
$\{x^{(\alpha)}\;|\;\alpha\in \ZZ^{(m,\underline{n})}\}$ where multiplication given by
\[
x^{(\alpha)}x^{(\beta)}=\binom{\alpha+\beta}{\alpha}x^{(\alpha+\beta)}\mbox{ where }\binom{\alpha+\beta}{\alpha}=\prod_{i=1}^m\binom{\alpha_i+\beta_i}{\alpha_i}.
\]
\end{df}

If $\ul{n}=\ul{1}$, then $\ca\cong\FF[x_1,\ld,x_m]/(x_1^p,\ld,x_m^p)$ by identifying $x_i$ with $x^{(\veps_i)}$:
\[
x^{(\alpha)}=\frac{1}{\alpha_1!\cdots\alpha_m!}x_1^{\alpha_1}\cdots x_m^{\alpha_m}.
\] 

The algebra $\ca$ has a canonical $\ZZ$-grading $\ca=\bigoplus_{\ell\geq 0}\ca_\ell$ defined by declaring the degree of $x^{(\alpha)}$ to be $|\alpha|$.
The associated filtration will be denoted by 
\[
\ca_{(\ell)}:=\bigoplus_{j\geq\ell}\ca_j.
\] 
Note that $\M:=\ca_{(1)}$ is the unique maximal ideal of $\ca$.

We will now define the following graded Cartan type Lie algebras: Witt, special and Hamiltonian. 
The contact algebras and generalized (i.e., non-graded) Cartan type Lie algebras will not be considered in this paper.

\begin{df}
Define a linear map $\D_i:\ca\to\ca$ by $\D_i x^{(\alpha)}=x^{(\alpha-\veps_i)}$ where the right-hand side is understood to be zero if $\alpha_i=0$. Then $\D_i$ is a derivation of $\ca$. 
The {\em Witt algebra} $W=\Wmn$ is the subalgebra of $\Der(\ca)$ that consists of all operators of the form
\[
f_1\D_1+\cdots+f_m\D_m\mbox{ where }f_i\in\ca.
\]
\end{df}
The canonical $\ZZ$-grading of $\ca$ induces a $\ZZ$-grading on $\End(\ca)$. Since $W$ is a graded subspace of $\End(\ca)$, it inherits the $\ZZ$-grading: 
$W=\bigoplus_{\ell\geq -1}W_\ell$. We will denote the associated filtration by $W_{(\ell)}$.

The de Rahm complex $\Omega^0\stackrel{d}{\longrightarrow}\Omega^1\stackrel{d}{\longrightarrow}\Omega^2\stackrel{d}{\longrightarrow}\ld$ is defined as follows: 
$\Omega^0=\ca$, $\Omega^1=\Hom_\ca(W,\ca)$, and $\Omega^k=(\Omega^1)^{\wedge k}$ for $k\geq 2$. 
The map $d:\Omega^0\to\Omega^1$ is defined by $(df)(D)=D(f)$ for all $f\in\ca$ and $D\in W$. The remaining maps $d:\Omega^k\to\Omega^{k+1}$ are defined in the usual way:
$d\left(f dx_{i_1}\wedge\cdots\wedge dx_{i_k}\right)=df\wedge dx_{i_1}\wedge\cdots\wedge dx_{i_k}$. 

Any element $D\in W$ acts on $\Omega^1=\Hom_\ca(W,\ca)$ by setting 
\[
D(\omega)(E)=D(\omega(E))-\omega([D,E])\mbox{ for all }\omega\in\Omega^1\mbox{ and }E\in W.
\]
This action turns all $\Omega^k=(\Omega^1)^{\wedge k}$ into $W$-modules. Of course, all $\Omega^k$ also have canonical $\ZZ$-gradings and associated filtrations.

We will need the following differential forms to define the special and Hamiltonian algebras:
\begin{align*}
\oS&:=dx_1\wedge dx_2\wedge\cdots\wedge dx_m\in\Omega^m,&(m\geq 2);\\
\oH&:=dx_1\wedge dx_{r+1}+dx_2\wedge dx_{r+2}+\cdots+dx_r\wedge dx_{2r}\in\Omega^2&(m=2r).
\end{align*}

\begin{df}
The {\em special algebra} $S=\Smn$ is the stabilizer of $\oS$ in $\Wmn$:
\[
S=\{D\in W\;|\;D(\oS)=0\}.
\]
The {\em Hamiltonian algebra} $H=\Hmn$ is the stabilizer of $\oH$ in $\Wmn$:
\[
H=\{D\in W\;|\;D(\oH)=0\}.
\]
\end{df}

Note that in the case $m=2$, we have $\oS=\oH$ and hence $S=H$. It is well-known that $\Wmn$ is simple unless $p=2$ and $m=1$. 
The algebras $\Smn$ and $\Hmn$ are not simple, but the first derived algebra $\Smn^{(1)}$ ($m\geq 3$) 
and the second derived algebra $\Hmn^{(2)}$ are simple. 

The Lie algebras $\Wmn$, $\Smn$ and $\Hmn$ are {\em restrictable} --- i.e., admit $p$-maps making them restricted Lie algebras --- if and only if $\ul{n}=\ul{1}$. 
From now on, we will assume that this is the case. Since $W$, $S$ and $H$ have trivial center, their $p$-maps are unique. 
Also, in this case $W=\Der(\ca)$ and hence the $p$-map is just the $p$-th power in the associative algebra $\End(\ca)$. 
The algebras $S$ and $H$ are restricted subalgebras of $W$.


\section{Automorphism group schemes}\label{section_schemes}

Let $\ca=\ca(m;\ul{1})$. Any automorphism $\mu$ of the algebra $\ca$ gives rise to an automorphism $\Ad(\mu)$ of $W$ given by $\Ad(\mu)(D)=\mu\circ D\circ\mu^{-1}$.
Then we can define the action of $\mu$ on $\Omega^1=\Hom_\ca(W,\ca)$ by setting $\mu(\omega)(D)=\mu(\omega(\Ad(\mu^{-1})(D)))$ for all $\omega\in\Omega^1$ and $D\in W$.
This action turns all $\Omega^k=(\Omega^1)^{\wedge k}$ into $\Aut(\ca)$-modules. Clearly, these actions can still be defined in the same way 
if we extend the scalars from the base field $\FF$ to any commutative associative $\FF$-algebra $K$, i.e., 
replace $\ca$ with $\ca(K):=\ca\ot K$, $W$ with $W(K):=W\ot K$ and $\Omega^k$ with $\Omega^k(K):=\Omega^k\ot K$.

Recall the automorphism group scheme of a finite-dimensional $\FF$-algebra $U$ (see e.g. \cite{Wh} for background on affine group schemes). 
As a functor, the (affine) group scheme $\AAut(U)$ is defined by setting $\AAut(U)(K)=\Aut_K(U\ot K)$ for any commutative associative $\FF$-algebra $K$.
From the above discussion it follows that we have morphisms of group schemes $\Ad:\AAut(\ca)\to\AAut(W)$ and also $\AAut(\ca)\to\GL(\Omega^k)$.
(We identify a smooth algebraic group scheme such as $\GL(V)$ with the corresponding algebraic group.) 
Note that, since the $p$-map of $W\ot K$ is uniquely determined, the automorphism group scheme of $W$ as a Lie algebra 
is the same as its automorphism group scheme as a restricted Lie algebra. Note also that the maps $d:\Omega^k\to\Omega^{k+1}$ are 
$\AAut(\ca)$-equivariant.

The algebraic group scheme $\AAut(U)$ contains the algebraic group $\Aut(U)$ as the largest smooth subgroupscheme. 
The tangent Lie algebra of $\AAut(U)$ is $\Der(U)$, so $\AAut(U)$ is smooth if and only if $\Der(U)$ equals the tangent Lie algebra of the group $\Aut(U)$.
The automorphism group schemes of simple Cartan type Lie algebras, unlike those of the classical simple Lie algebras, are not smooth. 
Indeed, the tangent Lie algebra of $\Aut(\ca)$ is $W_{(1)}$, which is a proper subalgebra of $W=\Der(\ca)$, so $\AAut(\ca)$ is not smooth. 
In view of the following theorem, we see that $\AAut(W)$ is not smooth.

\begin{theorem}[\cite{Wh_aut_1}]
\label{iso_W}
Let $\ca=\ca(m;\ul{1})$ and $W=\Wm$. Assume $p>3$. 
Then the morphism $\Ad:\AAut(\ca)\to\AAut(W)$ is an isomorphism of group schemes.\hfill{$\square$}
\end{theorem}

The automorphism group scheme of the general $\Wmn$ is also known --- see \cite{Wh_aut_2} ($p>2$, with small exceptions in the case $p=3$) 
and \cite{Sk1,Sk2} (any $p$, with small exceptions in the cases $p=2$ and $p=3$). In particular, Theorem \ref{iso_W} holds for $p=3$ if $m\geq 2$ and for $p=2$ if $m\geq 3$.
In this section we will establish analogues of Theorem \ref{iso_W} for the simple algebras $\Sm^{(1)}$ and $\Hm^{(2)}$. 
We will follow the approach of \cite{Wh_aut_1}. Suppose $\Phi:\Gs\to\Hs$ is a morphism of algebraic group schemes. 
Let $\Gs_{red}$ and $\Hs_{red}$ be the largest smooth subgroupschemes, which will be regarded as algebraic groups.
In order for $\Phi$ to be an isomorphism, the following two conditions are necessary:
\begin{enumerate}
\item[A)] the restriction $\Phi:\Gs_{red}\to\Hs_{red}$ is a bijection;
\item[B)] the tangent map $\Lie(\Phi):\Lie(\Gs)\to\Lie(\Hs)$ is a bijection. 
\end{enumerate}
However, unless $\Gs$ is known to be smooth (i.e., $\Gs=\Gs_{red}$), 
these two conditions are {\em not} sufficient for $\Phi$ to be an isomorphism. 
In general, one has to show that the associated map of distribution algebras $\wt{\Phi}:\Gd\to\Hd$ is surjective.
(The two conditions above imply that $\wt{\Phi}:\Gd\to\Hd$ is injective.) 

Recall that the distribution algebra $\Gd$ of an algebraic group scheme $\Gs$ is a connected cocommutative Hopf algebra 
(see e.g. \cite{Mont} for background on Hopf algebras), 
with the space of primitive elements $\Prim(\Gd)=\Lie(\Gs)$ of finite dimension. Hence $\Lie(\Gs)$ has a descending chain of restricted Lie subalgebras
(see e.g. \cite{Sw} or \cite[II, \S3, No. 2]{Dieu}):
\[
\Lie(\Gs)=\Lie^0(\Gs)\supset\Lie^1(\Gs)\supset\Lie^2(\Gs)\supset\ld,
\]
defined by $\Lie^k(\Gs):=\Lie^k(\Gd)=\V^k(\Gd)\cap\Prim(\Gd)$, where $\V:\Gd\to\Gd$ is the Verschiebung operator 
(see e.g. \cite[Theorem 1]{Sw} or \cite[II, \S2, No. 7]{Dieu}).
The intersection of this chain is $\Lie(\Gs_{red})$, which can be identified with the tangent algebra of the algebraic group $\Gs_{red}$. 

Recall that a sequence of elements $1=\sdp{0},\sdp{1},\ld,\sdp{n}$ in a connected cocommutative Hopf algebra $\Gd$ is called a {\em sequence of divided powers} 
(lying over $\sdp{1}$) if $\Delta(\sdp{j})=\sum_{i=0}^j \sdp{i}\ot\sdp{j-i}$ for all $j=1,\ld,n$. 
It follows that $\sdp{1}\in\Prim(\Gd)$ and that $\veps(\sdp{j})=0$ for all $j=1,\ld,n$.
It is easy to see that $\V^k(\sdp{p^k})=\sdp{1}$ for any $p^k\leq n$. Hence, if there exists a sequence of divided powers of length $p^k$ lying over $h\in\Prim(\Gd)$, 
then $h\in\Lie^k(\Gd)$. The converse is also true \cite[Theorem 2]{Sw}. 

Now we come back to the problem of proving that a morphism $\Phi:\Gs\to\Hs$ of algebraic group schemes is an isomorphism. 
Assuming that $\Phi$ satisfies the above two conditions, we need to show that the Hopf subalgebra $\wt{\Phi}(\Gd)\subset\Hd$ in fact equals $\Hd$.
Regarding $\Hd$ as a Hopf subalgebra in the distribution algebra of $\GL(V)$ for a suitable space $V$, we can apply \cite[II, \S3, No. 2, Corollary 1]{Dieu}
to conclude that $\wt{\Phi}(\Gd)=\Hd$ if and only if 
\begin{enumerate}
\item[C)] $\Lie(\Phi)$ maps $\Lie^k(\Gs)$ onto $\Lie^k(\Hs)$ for all $k$.
\end{enumerate}

Here we are interested in the case $\Gs=\AAut(\ca)$, where $\ca=\ca(m;\ul{1})$, and its subgroupschemes $\AAut_S(\ca):=\mathrm{Stab}_\Gs(\langle\oS\rangle)$ 
and $\AAut_H(\ca):=\mathrm{Stab}_\Gs(\langle\oH\rangle)$. We have 
\[
\begin{array}{ll}
\AAut(\ca)_{red}=\Aut(\ca), & \Lie(\AAut(\ca))=\Der(\ca)=W;\\
\AAut_S(\ca)_{red}=\mathrm{Stab}_{\Aut(\ca)}(\langle\oS\rangle), & \Lie(\AAut_S(\ca))=\mathrm{Stab}_W(\langle\oS\rangle)=:CS;\\
\AAut_H(\ca)_{red}=\mathrm{Stab}_{\Aut(\ca)}(\langle\oH\rangle), & \Lie(\AAut_H(\ca))=\mathrm{Stab}_W(\langle\oH\rangle)=:CH.
\end{array}
\]
We will denote $\mathrm{Stab}_{\Aut(\ca)}(\langle\oS\rangle)$ by $\Aut_S(\ca)$ and $\mathrm{Stab}_{\Aut(\ca)}(\langle\oH\rangle)$ by $\Aut_H(\ca)$ for brevity.

Assume $p>3$. It is known that the morphism $\Ad:\AAut(\ca)\to\AAut(W)$ as well as its restrictions $\AAut_S(\ca)\to\AAut(S^{(1)})$ ($m\geq 3$) and 
$\AAut_H(\ca)\to\AAut(H^{(2)})$ ($m=2r$) induce bijections $\Aut(\ca)\to\Aut(W)$, $\Aut_S(\ca)\to\Aut(S^{(1)})$ and $\Aut_H(\ca)\to\AAut(H^{(2)})$ 
--- see e.g. \cite[Theorem 7.3.2]{St}. Also, the tangent map $\ad:W\to\Der(W)$ as well as its restrictions $CS\to\Der(S^{(1)})$ and $CH\to\Der(H^{(2)})$ are bijective 
--- see e.g. \cite[Theorem 7.1.2]{St}. Thus, the above conditions A) and B) are satisfied for the morphisms $\AAut(\ca)\to\AAut(W)$, $\AAut_S(\ca)\to\AAut(S^{(1)})$ and 
$\AAut_H(\ca)\to\AAut(H^{(2)})$.

By \cite[Lemma 3.5, 2]{AS}, $\Lie^k(\AAut(\ca))=\mathrm{Stab}_W(\M)=W_{(0)}$ for all $k>0$. Hence, for $k>0$, we have  
\begin{align*}
\Lie^k(\AAut_S(\ca))&\subset\Lie^k(\AAut(\ca))\cap CS=CS_{(0)},\\
\Lie^k(\AAut_H(\ca))&\subset\Lie^k(\AAut(\ca))\cap CH=CH_{(0)}.
\end{align*}
On the other hand, $\Lie(\Aut_S(\ca))=CS_{(0)}$ and $\Lie(\Aut_H(\ca))=CH_{(0)}$. It follows that, in fact,
\begin{equation}
\label{Lie_k}
\Lie^k(\AAut_S(\ca))=CS_{(0)}\mbox{ and }\Lie^k(\AAut_H(\ca))=CH_{(0)}\mbox{ for all }k>0.
\end{equation}

By \cite[Lemma 3.5, 4]{AS}, $\Lie^k(\AAut(W))=\ad(W_{(0)})$ for all $k>0$, so condition C) is satisfied for the morphism $\AAut(\ca)\to\AAut(W)$. 
This is how Theorem \ref{iso_W} is proved in \cite{Wh_aut_1}. We are now ready to prove our analogues for $S$ and $H$.

\begin{theorem}
\label{iso_S}
Let $\ca=\ca(m;\ul{1})$, $m\geq 3$, and $S^{(1)}=\Sm^{(1)}$. Let 
\[
\AAut_S(\ca)=\mathrm{Stab}_{\AAut(\ca)}(\langle\oS\rangle).
\]
Assume $p>3$. Then the morphism $\Ad:\AAut_S(\ca)\to\AAut(S^{(1)})$ is an isomorphism of group schemes.
\end{theorem}

\begin{proof}
By the above discussion, we have to prove that condition C) is satisfied for $\Ad:\AAut_S(\ca)\to\AAut(S^{(1)})$. Let $\Gd$ be the distribution algebra of $\AAut(S^{(1)})$.
In view of (\ref{Lie_k}), it suffices to show that $\Lie^1(\Gd)\subset\ad(CS_{(0)})$. In other words, we have to verify, for any $D\in CS$, that if $D\notin CS_{(0)}$, 
then $\ad D\notin\Lie^1(\Gd)$. We can write $D=\lambda_1\D_1+\cdots+\lambda_m\D_m+D_0$ where $D_0\in CS_{(0)}$ and the scalars $\lambda_1,\ld\lambda_m$ are not all zero.
Now, $D_0\in\Lie^1(\AAut_S(\ca))$ implies $\ad D_0\in\Lie^1(\Gd)$, so it suffices to prove that $\ad(\lambda_1\D_1+\cdots+\lambda_m\D_m)\notin\Lie^1(\Gd)$. 
Applying an automorphism of $\ca$ induced by a suitable linear transformation on the space $\lspan{\D_1,\ld,\D_m}$, we may assume without loss of generality that $D=\D_1$. 

By way of contradiction, assume that $\ad\D_1\in\Lie^1(\Gd)$. Then there exists a sequence of divided powers $1=\sdp{0},\sdp{1},\ld,\sdp{p}$ in $\Gd$ such that 
$\sdp{1}=\ad \D_1$. 

As pointed out in the proof of \cite[Lemma 3.5, 4]{AS}, it follows from \cite[Lemma 7]{Sw} that we may assume without loss of generality that 
$\sdp{k}=\frac{1}{k!}(\sdp{1})^k$ for $k=0,\ld,p-1$. The distribution algebra $\Gd$ acts canonically on $S^{(1)}$, so we have a homomorphism $\eta:\Gd\to\End(S^{(1)})$. 
The restriction of $\eta$ to $\Prim(\Gd)=\Der(S^{(1)})$ is the identity map. Let $\sdpo{k}=\eta(\sdp{k})$ for $k=0,\ld,p$. Then we have 
\begin{equation}
\label{sdpo_k}
\sdpo{k}=\frac{1}{k!}(\ad \D_1)^k\mbox{ for }k=0,\ld,p-1,
\end{equation}
and, since $S^{(1)}$ is a $\Gd$-module algebra,
\begin{equation}
\label{prod_expansion_S}
\sdpo{k}([X,Y])=\sum_{j=0}^{k}\left[\sdpo{j}(X),\sdpo{k-j}(Y)\right]\mbox{ for all }k=0,\ld,p\mbox{ and }X,Y\in S^{(1)}.
\end{equation}

Note that the action of $\Gd$ on $S^{(1)}$ extends canonically to the universal enveloping algebra $U(S^{(1)})$ and, 
since the $p$-map of $S^{(1)}$ is uniquely determined, the $\Gd$-action passes on to the restricted enveloping algebra $u(S^{(1)})$. 
By abuse of notation, we will use $\sdpo{k}$ to denote the action of $\sdp{k}$ on $u(S^{(1)})$ as well as on $S^{(1)}$. 
We note for future reference that
\begin{equation}
\label{prod_expansion_uS}
\sdpo{k}(XY)=\sum_{j=0}^{k}(\sdpo{j})(X)\,(\sdpo{k-j})(Y)\mbox{ for all }k=0,\ld,p\mbox{ and }X,Y\in u(S^{(1)}).
\end{equation}

Observe that if we replace $\sdpo{p}$ by $\sdpo{p}+\xi$, where $\xi$ is any derivation of $S^{(1)}$, 
then equations (\ref{prod_expansion_S}) and (\ref{prod_expansion_uS}) will still hold (with the same $\sdpo{0},\ld,\sdpo{p-1}$). 
We will use this observation to simplify the operator $\sdpo{p}$. 

Let $z_i=1+x_i$, $i=1,\ld,m$. For each multi-index $\alpha\in\ZZ^{(m;\ul{1})}$, we set 
\[
z^{\alpha}=z_1^{\alpha_1}\cdots z_m^{\alpha_m}.
\]
Since $z_i^p=1$ for all $i$, we may regard the components of $\alpha$ as elements of the cyclic group $\ZZ_p$ when dealing with $z^\alpha$. It is this property 
that will make the basis $\{z^{\alpha}\}$ of $\ca=\ca(m;\ul{1})$ more convenient for us than the standard basis $\{x^{\alpha}\}$.  

Recall \cite[Section 4.2]{St} that $S^{(1)}=\Sm^{(1)}$ is spanned by the elements of the form 
\[
D_{i,j}(f):=\D_j(f)\D_i-\D_i(f)\D_j
\] 
where $f\in\ca$ and $1\le i<j\le m$. Hence, $S^{(1)}$ is spanned by the elements 
\[ 
D_{i,j}(z^{\alpha})=\alpha_j z^{\alpha-\veps_j}\D_i-\alpha_i z^{\alpha-\veps_i}\D_j.
\]

For the calculations we are about to carry out, we will need the following:

\begin{lemma}
For any $1\le i<j\le m$ and $\alpha,\beta\in\ZZ^{(m;\ul{1})}$, the commutator 
\[
\left[\DA{\alpha},\dij{i}{j}{\beta}\right]
\]
is given by the following expressions:
\begin{equation}\label{S_e1}
-(\alpha_1\beta_2-\alpha_2\beta_1)\dij{i}{j}{\alpha+\beta-\veps_1-\veps_2}\;\mbox{ if }2<i<j;
\end{equation}
\begin{equation}\label{S_e2}
-(\alpha_1(\beta_2-1)-\alpha_2\beta_1)\dij{2}{j}{\alpha+\beta-\veps_1-\veps_2}\;\mbox{ if }2=i<j;
\end{equation}
\begin{equation}\label{S_e3}
\begin{array}{r}
-\alpha_1\beta_j\DA{\alpha+\beta-\veps_1-\veps_j}+\alpha_2\beta_1\dij{1}{j}{\alpha+\beta-\veps_1-\veps_2}-\alpha_1\beta_1\dij{2}{j}{\alpha+\beta-2\veps_1}\\
\mbox{ if }i=1, j>2;
\end{array}
\end{equation}
\begin{equation}\label{S_e4}
-(\alpha_1\beta_2-\alpha_2\beta_1)\DA{\alpha+\beta-\veps_1-\veps_2}\mbox{ if }i=1, j=2.
\end{equation}
\end{lemma}

\begin{proof}
The verification of the above expressions is straightforward. Here we will verify (\ref{S_e3}), which is somewhat special, and leave the rest to the reader.
We have, for $j>2$, 
\begin{align*}
\left[\DA{\alpha},\dij{1}{j}{\beta}\right]=&
\left[\alpha_2 z^{\alpha-\veps_2}\D_1-\alpha_1 z^{\alpha-\veps_1}\D_2,\beta_j z^{\beta-\veps_j}\D_1-\beta_1 z^{\beta-\veps_1}\D_j\right]\\
=&\phantom{+}\alpha_2\beta_j\beta_1 z^{\alpha+\beta-\veps_1-\veps_2-\veps_j}\D_1 - \alpha_2\beta_j\alpha_1 z^{\alpha+\beta-\veps_1-\veps_2-\veps_j}\D_1\\
&+\alpha_2\beta_1\alpha_j z^{\alpha+\beta-\veps_1-\veps_2-\veps_j}\D_1 - \alpha_2\beta_1(\beta_1-1)z^{\alpha+\beta-2\veps_1-\veps_2}\D_j\\
&+\alpha_1\beta_j(\alpha_1-1)z^{\alpha+\beta-2\veps_1-\veps_j}\D_2 - \alpha_1\beta_j\beta_2 z^{\alpha+\beta-\veps_1-\veps_2-\veps_j}\D_1\\
&+\alpha_1\beta_1\beta_2 z^{\alpha+\beta-2\veps_1-\veps_2}\D_j - \alpha_1\beta_1\alpha_j z^{\alpha+\beta-2\veps_1-\veps_j}\D_2\\
=&(\alpha_2\beta_j\beta_1-\alpha_2\beta_j\alpha_1+\alpha_2\beta_1\alpha_j-\alpha_1\beta_j\beta_2)z^{\alpha+\beta-\veps_1-\veps_2-\veps_j}\D_1\\
 &+\alpha_1(\beta_j(\alpha_1-1)-\beta_1\alpha_j)z^{\alpha+\beta-2\veps_1-\veps_j}\D_2\\
 &+\beta_1(-\alpha_2(\beta_1-1)+\alpha_1\beta_2)z^{\alpha+\beta-2\veps_1-\veps_2}\D_j. 
\end{align*}

Comparing the above with 
\begin{align*} 
\DA{\alpha+\beta-\veps_1-\veps_j}=&(\alpha_2+\beta_2)z^{\alpha+\beta-\veps_1-\veps_2-\veps_j}\D_1-(\alpha_1+\beta_1-1)z^{\alpha+\beta-2\veps_1-\veps_j}\D_2,\\
\dij{1}{j}{\alpha+\beta-\veps_1-\veps_2}=&(\alpha_j+\beta_j)z^{\alpha+\beta-\veps_1-\veps_2-\veps_j}\D_1-(\alpha_1+\beta_1-1)z^{\alpha+\beta-2\veps_1-\veps_2}\D_j,\\
\dij{2}{j}{\alpha+\beta-2\veps_1}=&(\alpha_j+\beta_j)z^{\alpha+\beta-2\veps_1-\veps_j}\D_2-(\alpha_2+\beta_2)z^{\alpha+\beta-2\veps_1-\veps_2}\D_j,
\end{align*}
one readily sees that (\ref{S_e3}) holds.
\end{proof}

Another useful fact is the following:
\begin{equation}\label{partial_dij}
[\D_\ell,D_{i,j}(f)]=D_{i,j}(\D_\ell(f))\mbox{ for all }i,j,\ell=1,\ld,m\mbox{ and }f\in\ca.
\end{equation}

The remaining part of the proof of Theorem \ref{iso_S} will be divided into four steps. 

\medskip

{\bf Step 1}: Without loss of generality, we may assume
\begin{equation}\label{basis_induction_S}
\sdpo{p}\left(D_{1,2}(z_1z_2)\right)=0.
\end{equation}

Substituting $\alpha=\veps_1+\veps_2$ into expressions (\ref{S_e1}) through (\ref{S_e4}), 
we see that each nonzero element $\dij{i}{j}{\beta}$ is an eigenvector for the operator 
$\ad{D_{1,2}(z_1z_2)}$, with eigenvalue $\lambda(i,j,\beta)=\beta_1-\beta_2-1$, $\beta_1-\beta_2$ or $\beta_1-\beta_2+1$, depending on $i,j$.  
(For the case $i=1$ and $j>2$, one has to combine the first and the third terms in expression (\ref{S_e3}), which gives $-(\beta_2+1)\dij{1}{j}{\beta}$.)
Since $\lambda(i,j,\beta)$ is in the field $GF_p$, we have $D_{1,2}(z_1z_2)=D_{1,2}(z_1z_2)^p$. 
Applying the operator $\sdpo{p}$ to both sides and using (\ref{prod_expansion_uS}), we obtain:
\[
\sdpo{p}\left(D_{1,2}(z_1z_2)\right)=\sum_{i_1+\cdots+i_p=p}(\sdpo{i_1})\left(D_{1,2}(z_1z_2)\right) \cdots (\sdpo{i_p})\left(D_{1,2}(z_1z_2)\right)
\]
Taking into account (\ref{sdpo_k}) and (\ref{partial_dij}), we see that $\sdpo{k}\left(D_{1,2}(z_1z_2)\right)=0$ for $1<k<p$. Hence,
\begin{align}\label{p_power_S}
\sdpo{p}\left(D_{1,2}(z_1z_2)\right)=&\left(\sdpo{1}\left(D_{1,2}(z_1z_2)\right)\right)^p\\
&+\sum_{k=0}^{p-1}D_{1,2}(z_1z_2)^k(\sdpo{p})\left(D_{1,2}(z_1z_2)\right)D_{1,2}(z_1z_2)^{p-k-1}.\nonumber
\end{align}
Since $\sdpo{1}\left(D_{1,2}(z_1z_2)\right)=[\D_1,D_{1,2}(z_1z_2)]=D_{1,2}(z_2)=\D_1$ and $\D_1^p=0$, the first term on the right-hand side of (\ref{p_power_S}) vanishes. 
The second term can be rewritten using the identity
\[
\sum_{k=0}^{p-1}X^k\,Y X^{p-k-1}=(\ad{X})^{p-1}(Y),
\]
where $X=D_{1,2}(z_1z_2)$ and $Y=\sdpo{p}(X)$. Thus, (\ref{p_power_S}) yields
\[
Y=\left(\ad{D_{1,2}(z_1z_2)}\right)^{p-1}(Y).
\]
It follows that $Y$ can be written as a linear combination of those $\dij{i}{j}{\beta}$ for which the eigenvalue $\lambda(i,j,\beta)$ is nonzero:
\[
\sdpo{p}\left(D_{1,2}(z_1z_2)\right)=\sum_{i,j,\beta\,:\,\lambda(i,j,\beta)\neq 0}\sigma^{i,j}_\beta\dij{i}{j}{\beta}.
\]
It remains to replace $\sdpo{p}$ with 
\[
\sdpo{p}+\sum_{i,j,\beta}\frac{\sigma^{i,j}_\beta}{\lambda(i,j,\beta)}\,\ad{\dij{i}{j}{\beta}},
\]
to complete Step 1.

\medskip

{\bf Step 2}: Without loss of generality, we may assume that, in addition to (\ref{basis_induction_S}), we also have 
\begin{equation}\label{special_beta_S}
\DP(\D_1)=\sum_{i,j,\beta\,:\,\beta_1=p-1}\tau_{\beta}^{i,j}\dij{i}{j}{\beta}
\end{equation}
for some scalars $\tau_{\beta}^{i,j}$.

By (\ref{partial_dij}), we have $\left[\D_1,D_{1,2}(z_1z_2)\right]=\D_1$. Applying the operator $\DP$ to both sides and using (\ref{prod_expansion_S}), 
(\ref{sdpo_k}) and (\ref{basis_induction_S}), we obtain: 
\[
\DP(\D_1)=\sum_{k=0}^p\left[\sdpo{k}(\D_1),\sdpo{p-k}(D_{1,2}(z_1z_2))\right]=\left[\DP(\D_1),D_{1,2}(z_1z_2)\right].
\]
Hence $\left[D_{1,2}(z_1z_2),Y\right]=-Y$ where $Y=\DP(\D_1)$. It follows that $Y$ can be written as a linear combination of those $\dij{i}{j}{\beta}$ 
for which the eigenvalue $\lambda(i,j,\beta)$ is $-1$:
\begin{equation}\label{express_D1}
\DP(\D_1)=\sum_{i,j,\beta\,:\,\lambda(i,j,\beta)=-1}\tau^{i,j}_\beta\dij{i}{j}{\beta}.
\end{equation}
Now replace $\sdpo{p}$ with 
\[
\sdpo{p}+\sum_{i,j,\beta\,:\,\beta_1\neq p-1}\frac{\tau^{i,j}_\beta}{\beta_1+1}\,\ad{\dij{i}{j}{\beta+\veps_1}}.
\]
Using (\ref{partial_dij}), one readily sees that, for the new $\sdpo{p}$, we have (\ref{special_beta_S}), because all terms with $\beta_1\neq p-1$ 
in the right-hand side of (\ref{express_D1}) will cancel out. It remains to check that we still have (\ref{basis_induction_S}) for the new $\sdpo{p}$. 
In other words, we have to check that $[D_{1,2}(z_1z_2),\dij{i}{j}{\beta+\veps_1}]=0$ for all $i,j,\beta$ with $\tau^{i,j}_\beta\neq 0$.
But this is clear, because $\lambda(i,j,\beta+\veps_1)=\lambda(i,j,\beta)+1$ and thus we have 
$\lambda(i,j,\beta+\veps_1)=0$ for all $i,j,\beta$ that occur in the right-hand side of (\ref{express_D1}). Step 2 is complete.

\medskip

{\bf Step 3}: For any element $X=f_1\D_1+\cdots+f_m\D_m\in W$, define
\[
\pr(X):=f_1\in\ca.
\]
Assume (\ref{basis_induction_S}) and (\ref{special_beta_S}). Then, for any $k=1,\ld,p-1$, $\pr\left(\DP(\DOT{k})\right)$ is a linear combination of $z^\gamma$ with 
$0\le \gamma_1<k$.

We proceed by induction on $k$. The basis for $k=1$ follows from (\ref{basis_induction_S}). Now suppose the claim holds for some $k\ge 1$. By (\ref{partial_dij}), we have 
\[
[\D_1,\DOT{k+1}]=(k+1)\DOT{k}.
\] 
Applying the operator $\DP$ to both sides, we obtain
\begin{equation}\label{ey1}
\DP([\D_1,\DOT{k+1}])=(k+1)(\DP)(\DOT{k}).
\end{equation}
Using (\ref{prod_expansion_S}),(\ref{sdpo_k}) and (\ref{special_beta_S}), the left-hand side of (\ref{ey1}) becomes
\[
\left[\sum_{i,j,\beta\,:\,\beta_1=p-1}\tau_{\beta}^{i,j}\dij{i}{j}{\beta},\DOT{k+1}\right]+\left[\D_1,\,\DP(\DOT{k+1})\right].
\]
Setting $Y=\DP(\DOT{k+1})$, we can rewrite (\ref{ey1}) as follows:
\begin{equation}\label{ey2}
[\D_1,Y]=(k+1)(\DP)(\DOT{k})
+\sum_{i,j,\beta\,:\,\beta_1=p-1}\tau_{\beta}^{i,j}\left[\DOT{k+1},\dij{i}{j}{\beta}\right].
\end{equation}

Our goal is to show that monomials $z^\gamma$ with $\gamma_1\ge k+1$ do not occur in $f_1:=\pr(Y)$. Since $\pr[\D_1,Y]=\D_1 f_1$, it suffices to show that 
elements $z^\gamma \D_1$ with $\gamma_1\ge k$ do not occur in the right-hand side of (\ref{ey2}), when it is regarded as an element of $W$. 
The induction hypothesis tells us that such elements do not occur in the first term of the right-hand side of (\ref{ey2}). 
We are going to prove the same for the second term. 

In the case $2<i<j$, we have by (\ref{S_e1}):
\[
\left[\DOT{k+1},\dij{i}{j}{\beta}\right]=(\beta_1-\beta_2(k+1))\dij{i}{j}{\beta+k\ve_1}.
\]
Hence, no elements $z^\gamma\D_1$ occur here.

In the case $2=i<j$, we have by (\ref{S_e2}):
\[
\left[\DOT{k+1},\dij{i}{j}{\beta}\right]=(\beta_1-(\beta_2-1)(k+1))\dij{2}{j}{\beta+k\ve_1}.
\]
Again, no elements $z^\gamma\D_1$ occur.

In the case $i=1$ and $2<j$, we can write $\left[\DOT{k+1},\dij{i}{j}{\beta}\right]$, using (\ref{S_e3}) and $\beta_1=p-1$, as follows:
\[
-(k+1)\beta_j\DA{\beta+k\ve_1+\ve_2-\ve_j}-\dij{1}{j}{\beta+k\ve_1}+(k+1)\dij{2}{j}{\beta+(k-1)\ve_1+\ve_2}.
\]
Elements $z^\gamma\D_1$ occur only in the first two summands, and we have $\gamma=\beta+k\veps_1-\veps_j$ in either case. Therefore, $\gamma_1=\beta_1+k=k-1$ mod $p$ 
(recall that we may take the exponents of $z$ modulo $p$).

In the case $i=1$,$j=2$, we have by $\beta_1=p-1$ and (\ref{S_e4}): 
\[
\left[\DOT{k+1},\dij{1}{2}{\beta}\right]=-((k+1)\beta_2+1)\DA{\beta+k\ve_1}.
\]
Hence, elements $z^\gamma\D_1$ occur with $\gamma=\beta+k\veps_1-\veps_2$. Once again, $\gamma_1=k-1$ mod $p$.

The inductive proof of Step 3 is complete.

\medskip

{\bf Step 4}: We can finally obtain a contradiction. By (\ref{S_e4}), we have 
\[
\left[\DOT{p-1},\DOT{2}\right]=3D_{1,2}(z_2)=3\D_1.
\]
Applying $\DP$ and taking into account (\ref{prod_expansion_S}), (\ref{sdpo_k}) and (\ref{special_beta_S}), we obtain:
\begin{align}\label{ey3}
3&\sum_{i,j,\beta\,:\,\beta_1=p-1}\tau_{\beta}^{i,j}\dij{i}{j}{\beta}\\
=&\left[\DP(\DOT{p-1}),\DOT{2}\right]+\left[\DOT{p-1},\DP(\DOT{2})\right]\nonumber \\
&+\left[D_{1,2}\left(\frac{1}{(p-1)!}\D_1^{p-1}(z_1^{p-1}z_2)\right),D_{1,2}\left(\frac{1}{1!}\D_1(z_1^2z_2)\right)\right]\nonumber \\
&+\left[D_{1,2}\left(\frac{1}{(p-2)!}\D_1^{p-2}(z_1^{p-1}z_2)\right),D_{1,2}\left(\frac{1}{2!}\D_1^2(z_1^2z_2)\right)\right].\nonumber
\end{align}
One readily verifies that the sum of the third and fourth terms in the right-hand side of (\ref{ey3}) is $3\D_1$. 
Let us consider the first and the second terms. Denote 
\[
X:=\DP(\DOT{p-1})\mbox{ and }Y:=\DP(\DOT{2}).
\] 
By Step 3, we know that $\pr(X)$ is a linear combination of $z^\gamma$ with $0\le\gamma_1\le p-2$ and $\pr(Y)$ is a linear combination of 
$z^\gamma$ with $0\le\gamma_1\le 1$. 

Since $\DOT{2}=z_1^2\D_1-2z_1 z_2\D_2$, we see that the coefficient of $\D_1$ depends only on $z_1$ and hence all 
terms with $\D_1$ in the commutator $[X,\DOT{2}]$ come from terms with $\D_1$ in $X$. In other words,
\[
\pr[X,\DOT{2}]=\pr[\pr(X)\D_1,\DOT{2}].
\]
Since
\[
[z^\gamma\D_1,\DOT{2}]=(2-\gamma_1+2\gamma_2)z^{\gamma+\veps_1}\D_1-2z^{\gamma+\veps_2}\D_2,
\]
we conclude that $\pr[X,\DOT{2}]$ is a linear combination of monomials $z^{\gamma+\veps_1}$ with $0\le\gamma_1\le p-2$.
Therefore, elements $z^\alpha\D_1$ with $\alpha_1=0$ do not occur in the first term in the right-hand side of (\ref{ey3}).

Since $\DOT{p-1}=z_1^{p-1}\D_1+z_1^{p-2}z_2\D_2$, we also see that 
\[
\pr[\DOT{p-1},Y]=\pr[\DOT{p-1},\pr(Y)\D_1].
\]
Since
\[
[\DOT{p-1},z^\gamma\D_1]=(1+\gamma_1+\gamma_2)z^{\gamma+(p-2)\veps_1}\D_1+2z^{\gamma+(p-3)\veps_1+\veps_2}\D_2,
\]
we conclude that $\pr[\DOT{p-1},Y]$ is a linear combination of monomials $z^{\gamma+(p-2)\veps_1}$ with $0\le\gamma_1\le 1$.
Therefore, elements $z^\alpha\D_1$ with $\alpha_1=0$ do not occur in the second term in the right-hand side of (\ref{ey3}).

Finally, all elements $z^\alpha\D_1$ that occur in the left-hand side of (\ref{ey3}) have $\alpha_1=p-1$. Summarizing our analysis,
we obtain
\[
3\D_1=0,
\]
which is a contradiction, since $p>3$. The proof of Theorem \ref{iso_S} is complete.
\end{proof}

\begin{corollary}
Under the assumptions of Theorem \ref{iso_S}, 
\[
\Lie^k(\AAut(S^{(1)}))=\ad(CS_{(0)})\mbox{ for all }k>0.
\]
\end{corollary}

\begin{theorem}
\label{iso_H}
Let $\ca=\ca(m;\ul{1})$, $m=2r$, and $H^{(2)}=\Hm^{(2)}$. Let 
\[
\AAut_H(\ca)=\mathrm{Stab}_{\AAut(\ca)}(\langle\oH\rangle). 
\]
Assume $p>3$. Then the morphism $\Ad:\AAut_H(\ca)\to\AAut(H^{(2)})$ is an isomorphism of group schemes.
\end{theorem}

\begin{proof}
Let $\Gd$ be the distribution algebra of $\AAut(H^{(2)})$.
As in the proof of Theorem \ref{iso_S}, we have to show that $\Lie^1(\Gd)\subset\ad(CH_{(0)})$ 
--- i.e., for any $D\in CH$, if $D\notin CH_{(0)}$, then $\ad D\notin\Lie^1(\Gd)$. 
Again, we can write $D=\lambda_1\D_1+\cdots+\lambda_m\D_m+D_0$ where $D_0\in CH_{(0)}$ and the scalars $\lambda_1,\ld\lambda_m$ are not all zero.
Since $D_0\in\Lie^1(\AAut_H(\ca))$ implies $\ad D_0\in\Lie^1(\Gd)$, it suffices to prove that $\ad(\lambda_1\D_1+\cdots+\lambda_m\D_m)\notin\Lie^1(\Gd)$. 
Applying an automorphism of $\ca$ induced by a suitable symplectic transformation on the space $\lspan{\D_1,\ld,\D_m}$, 
we may assume without loss of generality that $D=\D_1$. 

By way of contradiction, assume that $\ad\D_1\in\Lie^1(\Gd)$. Then there exists a sequence of divided powers $1=\sdp{0},\sdp{1},\ld,\sdp{p}$ in $\Gd$ such that 
$\sdp{1}=\ad \D_1$. We may assume without loss of generality that $\sdp{k}=\frac{1}{k!}(\sdp{1})^k$ for $k=0,\ld,p-1$. 

The distribution algebra $\Gd$ acts canonically on $H^{(2)}$, so we have a homomorphism $\eta:\Gd\to\End(H^{(2)})$. 
The restriction of $\eta$ to $\Prim(\Gd)=\Der(H^{(2)})$ is the identity map. Let $\sdpo{k}=\eta(\sdp{k})$ for $k=0,\ld,p$. Then we have
\begin{equation}
\label{sdpo_k_repeat}
\sdpo{k}=\frac{1}{k!}(\ad \D_1)^k\mbox{ for }k=0,\ld,p-1,
\end{equation} 
and, since $H^{(2)}$ is a $\Gd$-module algebra,
\begin{equation}
\label{prod_expansion_H}
\sdpo{k}([X,Y])=\sum_{j=0}^{k}[\sdpo{j}(X),\sdpo{k-j}(Y)]\mbox{ for all }k=0,\ld,p\mbox{ and }X,Y\in H^{(2)}.
\end{equation}
For the extended action of $\Gd$ on the restricted enveloping algebra $u(H^{(2)})$, we have  
\begin{equation}
\label{prod_expansion_uH}
\sdpo{k}(XY)=\sum_{j=0}^{k}(\sdpo{j})(X)\,(\sdpo{k-j})(Y)\mbox{ for all }k=0,\ld,p\mbox{ and }X,Y\in u(H^{(2)}).
\end{equation}

We may replace $\sdpo{p}$ by $\sdpo{p}+\xi$, where $\xi$ is any derivation of $H^{(2)}$, without affecting equations (\ref{prod_expansion_H}) and 
(\ref{prod_expansion_uH}). We will use this observation to simplify the operator $\sdpo{p}$. 

As in \cite[Section 4.2]{St}, we define 
\[
\begin{array}{lll}
\sigma(i):=\left\{\begin{array}{ll}1&\mbox{ if }i=1,\ld,r;\\ -1&\mbox{ if }i=r+1,\ld,2r;\end{array}\right.
& \mbox{ and } & i':=i+\sigma(i)r.
\end{array}
\]
Also we define a map $D_H:\ca\to H$ as follows:
\[
D_H(f):=\sum_{i=1}^{2r}\sigma(i)\D_i(f)\D_{i'}.
\]
Note that the kernel of $D_H$ is $\FF 1$. 

Let $z_i=1+x_i$, $i=1,\ld,2r$. For each multi-index $\alpha\in\ZZ^{(2r;\ul{1})}$, we set 
\[
z^{\alpha}=z_1^{\alpha_1}\cdots z_{2r}^{\alpha_{2r}}.
\]
We may regard the components of $\alpha$ as elements of the cyclic group $\ZZ_p$ when dealing with $z^\alpha$.
For the calculations we are about to carry out, we will need the following formula:
\begin{equation}
\label{H_e1}
[\da{\alpha},\da{\beta}]=D_H\left(\da{\alpha}(z^{\beta})\right)\mbox{ for all }\alpha,\beta\in\ZZ^{(2r;\ul{1})}.
\end{equation}
In particular,
\begin{equation}
\label{H_e2}
[\D_\ell,\da{\beta}]=D_H\left(\D_\ell(z^\beta)\right)\mbox{ for all }\ell=1,\ld,2r\mbox{ and }\beta\in\ZZ^{(2r;\ul{1})}.
\end{equation}  

Let $\tau=(p-1,\ld,p-1)\in\ZZ^{(2r;\ul{1})}$. It follows from \cite[Section 4.2]{St} that 
\[
\{D_H(z^\alpha)\;|\;0<\alpha<\tau\}\mbox{ is a basis of }H^{(2)}.
\]

The remaining part of the proof of Theorem \ref{iso_H} will be divided into four steps, 
which are similar to the steps in the proof of Theorem \ref{iso_S}.

\medskip

{\bf Step 1}: Without loss of generality, we may assume
\begin{equation}\label{basis_induction_H}
\sdpo{p}\left(D_H(z_1\zr)\right)=0.
\end{equation}

Substituting $\alpha=\veps_1+\veps_2$ into equation (\ref{H_e1}), we see that $\da{\beta}$ is an eigenvector for the operator 
$\ad{D_H(z_1\zr)}$, with eigenvalue $\beta_{r+1}-\beta_1$.  
Since $\beta_{r+1}-\beta_1$ is in the field $GF_p$, we have $D_H(z_1\zr)=D_H(z_1\zr)^p$. 
Applying the operator $\sdpo{p}$ to both sides and using (\ref{prod_expansion_uH}), we obtain:
\[
\sdpo{p}\left(D_H(z_1\zr)\right)=\sum_{i_1+\cdots+i_p=p}(\sdpo{i_1})\left(D_H(z_1\zr)\right) \cdots (\sdpo{i_p})\left(D_H(z_1\zr)\right)
\]
Taking into account (\ref{sdpo_k_repeat}) and (\ref{H_e2}), we see that $\sdpo{k}\left(D_H(z_1\zr)\right)=0$ for $1<k<p$. Hence,
\begin{align}\label{p_power_H}
\sdpo{p}\left(D_H(z_1\zr)\right)=&\left(\sdpo{1}\left(D_H(z_1\zr)\right)\right)^p\\
&+\sum_{k=0}^{p-1}D_H(z_1\zr)^k(\sdpo{p})\left(D_H(z_1\zr)\right)D_H(z_1\zr)^{p-k-1}.\nonumber
\end{align}
Since $\sdpo{1}\left(D_H(z_1\zr)\right)=[\D_1,D_H(z_1\zr)]=D_H(\zr)=-\D_1$ and $\D_1^p=0$, the first term on the right-hand side of (\ref{p_power_H}) vanishes. 
The second term can be rewritten using the identity
\[
\sum_{k=0}^{p-1}X^k\,Y X^{p-k-1}=(\ad{X})^{p-1}(Y),
\]
where $X=D_H(z_1\zr)$ and $Y=\sdpo{p}(X)$. Thus, (\ref{p_power_H}) yields
\[
Y=\left(\ad{D_H(z_1\zr)}\right)^{p-1}(Y).
\]
It follows that $Y$ is a linear combination of those $\da{\beta}$ for which the eigenvalue $\beta_{r+1}-\beta_1$ is nonzero:
\[
\sdpo{p}\left(D_H(z_1\zr)\right)=\sum_{\beta\,:\,\beta_{r+1}\neq\beta_1}\sigma_\beta\da{\beta}.
\]
It remains to replace $\sdpo{p}$ with 
\[
\sdpo{p}+\sum_{\beta}\frac{\sigma_\beta}{\beta_{r+1}-\beta_1}\,\ad{\da{\beta}},
\]
to complete Step 1.

\medskip

{\bf Step 2}: Without loss of generality, we may assume that, in addition to (\ref{basis_induction_H}), we also have 
\begin{equation}\label{special_beta_H}
\DP(\D_1)=\sum_{\beta\,:\,\beta_1=p-1}\tau_\beta\da{\beta}
\end{equation}
for some scalars $\tau_\beta$.

By (\ref{H_e2}), we have $\left[\D_1,D_H(z_1\zr)\right]=-\D_1$. Applying the operator $\DP$ to both sides and using (\ref{prod_expansion_H}), 
(\ref{sdpo_k_repeat}) and (\ref{basis_induction_H}), we obtain: 
\[
-\DP(\D_1)=\sum_{k=0}^p\left[\sdpo{k}(\D_1),\sdpo{p-k}(D_H(z_1\zr))\right]=\left[\DP(\D_1),D_H(z_1\zr)\right].
\]
Hence $\left[D_H(z_1\zr),Y\right]=Y$ where $Y=\DP(\D_1)$. It follows that $Y$ is a linear combination of those $\da{\beta}$ 
for which the eigenvalue $\beta_{r+1}-\beta_1$ is $1$:
\begin{equation}\label{H_e3}
\DP(\D_1)=\sum_{\beta\,:\,\beta_{r+1}-\beta_1=1}\tau_\beta\da{\beta}.
\end{equation}
Now replace $\sdpo{p}$ with 
\[
\sdpo{p}+\sum_{\beta\,:\,\beta_1\neq p-1}\frac{\tau_\beta}{\beta_1+1}\,\ad{\da{\beta+\veps_1}}.
\]
Using (\ref{H_e2}), one readily sees that, for the new $\sdpo{p}$, we have (\ref{special_beta_H}), because all terms with $\beta_1\neq p-1$ 
in the right-hand side of (\ref{H_e3}) will cancel out. It remains to check that we still have (\ref{basis_induction_H}) for the new $\sdpo{p}$. 
In other words, we have to check that $[D_H(z_1\zr),\da{\beta+\veps_1}]=0$ for all $\beta$ with $\tau_\beta\neq 0$.
But this is clear, because we have $\beta_{r+1}-(\beta_1+1)=0$ for all $\beta$ that occur in the right-hand side of (\ref{H_e3}). Step 2 is complete.

\medskip

{\bf Step 3}: Assume (\ref{basis_induction_H}) and (\ref{special_beta_H}). Then, for any $k=1,\ld,p-1$, $\DP(\dr{k})$ is a linear combination of $\da{\gamma}$ 
with $0\le\gamma_1<k$.

We proceed by induction on $k$. The basis for $k=1$ follows from (\ref{basis_induction_H}). Now suppose the claim holds for some $k\ge 1$. By (\ref{H_e2}), we have
\[
[\D_1,\dr{k+1}]=(k+1)\dr{k}.
\]
Applying the operator $\DP$ to both sides and taking into account (\ref{prod_expansion_H}) and (\ref{sdpo_k_repeat}), we obtain
\begin{equation}
\label{H_e4}
[\DP(\D_1),\dr{k+1}]+[\D_1,\,\DP(\dr{k+1})]=(k+1)(\DP)(\dr{k}).
\end{equation}
Writing
\begin{equation}
\label{express_DP_H}
\DP(\dr{k+1})=\sum_{0<\gamma<\tau}\sigma_\gamma\da{\gamma}
\end{equation}
and taking into account (\ref{special_beta_H}), (\ref{H_e1}) and (\ref{H_e2}), we can rewrite the left-hand side of (\ref{H_e4}) as follows:
\begin{align*}
&\sum_{\beta\,:\,\beta_1=p-1}\tau_\beta[\da{\beta},\dr{k+1}]+\sum_{\gamma}\sigma_\gamma [\D_1,\da{\gamma}]\\
&=-\sum_{\beta\,:\,\beta_1=p-1}\tau_\beta D_H(((k+1)z_1^k\zr\D_{r+1}-z_1^{k+1}\D_1)(z^{\beta}))+\sum_{\gamma}\sigma_\gamma D_H(\D_1(z^{\gamma}))\\
&=-\sum_{\beta\,:\,\beta_1=p-1}\tau_\beta D_H((k+1)\beta_{r+1}z^{\beta+k\ve_1}-\beta_1 z^{\beta+k\ve_1})+\sum_{\gamma}\sigma_\gamma\gamma_1\da{\gamma-\ve_1}\\
&=-\sum_{\beta\,:\,\beta_1=p-1}\tau_\beta ((k+1)\beta_{r+1}+1)\da{\beta+k\ve_1}+\sum_{\gamma}\sigma_\gamma\gamma_1\da{\gamma-\ve_1}.
\end{align*}
Setting $\tau'_\beta:=\tau_\beta ((k+1)\beta_{r+1}+1)$, we can now rewrite (\ref{H_e4}) as follows:
\begin{equation}
\label{H_e5}
\sum_{\gamma}\sigma_\gamma\gamma_1\da{\gamma-\ve_1}=(k+1)(\DP)(\dr{k})+\sum_{\beta\,:\,\beta_1=p-1}\tau'_\beta \da{\beta+k\ve_1}.
\end{equation}
Since the induction hypothesis applies to the first term in the right-hand side of (\ref{H_e5}), we see that all elements $\da{\alpha}$ that occur in the right-hand side
have $0\le\alpha_1\le k-1$. Comparing this with the left-hand side, we conclude that for any $\gamma$ with $\sigma_\gamma\ne 0$ we have either $\gamma_1=0$ or 
$0\le\gamma_1-1\le k-1$. Hence the elements $\da{\gamma}$ that occur in the right-hand side of (\ref{express_DP_H}) have $0\le\gamma_1\le k$.

The inductive proof of Step 3 is complete.

\medskip

{\bf Step 4}: We can finally obtain a contradiction. By (\ref{H_e1}), we have
\[
[\dr{2},\dr{p-1}]=3D_H(\zr)=-3\D_1.
\]
Applying $\DP$ and taking into account (\ref{prod_expansion_H}) and (\ref{sdpo_k_repeat}), we obtain:
\begin{align*}
-3(\DP)(\D_1)=&[\DP(\dr{2}),\dr{p-1}]+[\dr{2},\,\DP(\dr{p-1})]\\
&+\left[D_H\left(\frac{1}{1!}\D_1(z_1^2\zr)\right),D_H\left(\frac{1}{(p-1)!}\D_1^{p-1}(z_1^{p-1})\zr\right)\right]\\
&+\left[D_H\left(\frac{1}{2!}\D_1^2(z_1^2\zr)\right),D_H\left(\frac{1}{(p-2)!}\D_1^{p-2}(z_1^{p-1})\zr\right)\right]\\
=&[\DP(\dr{2}),\dr{p-1}]+[\dr{2},\,\DP(\dr{p-1})]\\
&+3[D_H(z_1\zr),D_H(\zr)]\\
=&[\DP(\dr{2}),\dr{p-1}]+[\dr{2},\,\DP(\dr{p-1})]\\
&-3\D_1.
\end{align*}
So we have 
\begin{align}
\label{H_e6}
3\D_1=&[\DP(\dr{2}),\dr{p-1}]+[\dr{2},\,\DP(\dr{p-1})]\\
&+3(\DP)(\D_1).\nonumber
\end{align}
By Step 3, we know that
\begin{align*}
\DP(\dr{2})&=\sum_{\alpha\,:\,0\le\alpha_1\le 1}\sigma^{(2)}_\alpha\da{\alpha};\\
\DP(\dr{p-1})&=\sum_{\alpha\,:\,0\le\alpha_1\le p-2}\sigma^{(p-1)}_\alpha\da{\alpha}.
\end{align*}
Hence, by (\ref{H_e1}) we obtain:
\begin{align*}
&[\DP(\dr{2}),\dr{p-1}]\\
&=\sum_{\alpha\,:\,0\le\alpha_1\le 1}\sigma^{(2)}_\alpha[\da{\alpha},\dr{p-1}]\\
&=-\sum_{\alpha\,:\,0\le\alpha_1\le 1}\sigma^{(2)}_\alpha D_H((-z_1^{p-2}\zr\D_{r+1}-z_1^{p-1}\D_1)(z^{\alpha}))\\
&=\sum_{\alpha\,:\,0\le\alpha_1\le 1}\sigma^{(2)}_\alpha (\alpha_{r+1}+\alpha_1)\da{\alpha-2\ve_1},
\end{align*}
and, similarly,
\begin{align*}
&[\dr{2},\,\DP(\dr{p-1})]\\
&=\sum_{\alpha\,:\,0\le\alpha_1\le p-2}\sigma^{(p-1)}_\alpha[\dr{2},\da{\alpha}]\\
&=\sum_{\alpha\,:\,0\le\alpha_1\le p-2}\sigma^{(p-1)}_\alpha D_H((2z_1\zr\D_{r+1}-z_1^2\D_1)(z^\alpha))\\
&=\sum_{\alpha\,:\,0\le\alpha_1\le p-2}\sigma^{(p-1)}_\alpha (2\alpha_{r+1}-\alpha_1)\da{\alpha+\ve_1}.
\end{align*}
Using the above calculations and (\ref{special_beta_H}), we can rewrite (\ref{H_e6}) as follows:
\begin{align*}
3\D_1=&\sum_{\alpha\,:\,0\le\alpha_1\le 1}\sigma^{(2)}_\alpha (\alpha_{r+1}+\alpha_1)\da{\alpha-2\ve_1}\\
&+\sum_{\alpha\,:\,0\le\alpha_1\le p-2}\sigma^{(p-1)}_\alpha (2\alpha_{r+1}-\alpha_1)\da{\alpha+\ve_1}\\
&+3\sum_{\beta\,:\,\beta_1=p-1}\tau_\beta\da{\beta}.
\end{align*}
This equation is impossible, because $p>3$ and none of the sums in the right-hand side involves $D_H(\zr)$.
The proof of Theorem \ref{iso_H} is complete.
\end{proof}

\begin{corollary}
Under the assumptions of Theorem \ref{iso_H}, 
\[
\Lie^k(\AAut(H^{(2)}))=\ad(CH_{(0)})\mbox{ for all }k>0.
\]
\end{corollary}


\section{Group gradings}\label{section_gradings}

Let $G$ be an abelian group. In this section we will give a classification of $G$-gradings on the algebra $\ca=\ca(m;\ul{1})$ and 
on the simple Lie algebras $\Wm$, $\Sm^{(1)}$ ($m\ge 3$) and $S(2;\ul{1})^{(2)}=H(2;\ul{1})^{(2)}$.  

Given a $G$-grading $\Gamma_\ca:\ca=\bigoplus_{g\in G}\ca_g$, we obtain an induced grading on $\End(\ca)$. It is easy to see that $W=\Der(\ca)$ is a graded subspace, 
so it inherits a $G$-grading, which will be denoted by $\Gamma_W:W=\bigoplus_{g\in G}W_g$. 
The spaces $\Omega^k$ also receive $G$-gradings in a natural way, and one can verify that the maps $d:\Omega^k\to\Omega^{k+1}$ respect the $G$-gradings.
(The canonical $\ZZ$-gradings of $W$ and $\Omega^k$ are induced by the canonical $\ZZ$-grading of $\ca$ in this manner.) However, $S=\Sm$ (respectively, $H=\Hm$) 
is not a graded subspace of $W$, in general. It is certainly a graded subspace if we assume that $\oS$ (respectively, $\oH$) is a homogeneous element with respect to 
the $G$-grading on $\Omega^m$ (respectively, $\Omega^2$).

\begin{df}
\label{X_admissible}
We will say that a $G$-grading $\Gamma_\ca:\ca=\bigoplus_{g\in G}\ca_g$ is $S$-{\em admissible} (respectively, $H$-{\em admissible}) {\em of degree} $g_0\in G$ 
if the form $\oS$ (respectively, $\oH$) is a homogeneous element of degree $g_0$.
\end{df}

If $\Gamma_\ca$ is $S$-admissible (respectively, $H$-admissible), then we will denote the induced $G$-gradings on $S$ (respectively, $H$) and its derived algebra(s) by 
$\Gamma_S$ (respectively, $\Gamma_H$).

We will now recall the connection between group gradings on an algebra and certain subgroupschemes of its automorphism group scheme. Let $U$ be an algebra. 
For any group $G$, a $G$-grading on $U$ is equivalent to a structure of an $\FF G$-comodule algebra (see e.g. \cite{Mont}). Assuming $U$ finite-dimensional 
and $G$ finitely generated abelian, we can regard this comodule structure as a morphism of algebraic group schemes $G^D\to\AAut(U)$ where $G^D$ is the Cartier dual 
of $G$, i.e., the group scheme represented by the commutative Hopf algebra $\FF G$. 
Two $G$-gradings are isomorphic if and only if the corresponding morphisms $G^D\to\AAut(U)$ are conjugate by an automorphism of $U$. 

If $\chr{\FF}=0$, then $G^D=\wh{G}$, the algebraic group of multiplicative characters of $G$, and $\AAut(U)=\Aut(U)$, the algebraic group of automorphisms.
The image of $\wh{G}$ in $\Aut(U)$ is a {\em quasitorus}, i.e., a diagonalizable algebraic group. The $G$-grading on $U$ is, of course, the eigenspace decomposition of 
$U$ with respect to this quasitorus. Hence, group gradings on $U$ correspond to quasitori in $\Aut(U)$.

Here we are interested in the case $\chr{\FF}=p>0$. Then we can write $G=G_{p'}\times G_p$ where $G_{p'}$ has no $p$-torsion and $G_p$ is a $p$-group. 
Hence $G^D=\wh{G_{p'}}\times G_p^D$, where $\wh{G_{p'}}$ is smooth and $G_p^D$ is finite and connected. 
The algebraic group $\wh{G_{p'}}$ (which is equal to $\wh{G}$) is a quasitorus, and it acts by automorphisms of $U$ as follows: 
\[
\chi\ast X=\chi(g)X\quad\mbox{for all}\;X\in U_g\;\mbox{and}\;g\in G. 
\]
If $G_p$ is an elementary $p$-group, then the distribution algebra of $G_p^D$ is the restricted enveloping algebra $u(T)$ where $T$ is the group of additive characters 
of $G_p$, regarded as an abelian restricted Lie algebra. If $\{a_1,\ld,a_s\}$ is a basis of $G_p$ (as a vector space over the field $GF_p$), then the dual basis 
$\{t_1,\ld,t_s\}$ of $T$ has the property $(t_i)^p=t_i$ for all $i$. Therefore, $T$ is a {\em torus} in the sense of restricted Lie algebras. 
It acts by derivations of $U$ as follows:
\[
t\ast X=t(g)X\quad\mbox{for all}\;X\in U_g\;\mbox{and}\;g\in G. 
\]
If $G_p$ is not elementary, then the distribution algebra of $G_p^D$ is not generated by primitive elements and hence its action on $U$ does not reduce to derivations.
Regardless of what the case may be, the image of $G^D$ in $\AAut(U)$ is a diagonalizable subgroupscheme. 
In some sense, the $G$-grading on $U$ is its eigenspace decomposition (see e.g. \cite{Wh}).

Now, Theorems \ref{iso_W} (where, as was pointed out, one can include the cases $p=2$ and $p=3$), \ref{iso_S} and \ref{iso_H} give us the following corollary.

\begin{corollary}
\label{gradings_come_from_O}
Let $G$ be an abelian group. Let $L$ be one of the following simple Lie algebras: $W=\Wm$ ($m\geq 3$ if $p=2$ and $m\geq 2$ if $p=3$), 
$S^{(1)}=\Sm^{(1)}$ ($m\geq 3$ and $p>3$) or $H^{(2)}=\Hm^{(2)}$ ($m=2r$ and $p>3$). Then any $G$-grading on $L$ is induced by a $G$-grading on $\ca=\ca(m;\ul{1})$.
More precisely,
\begin{enumerate}
\item[1)] The correspondence $\Gamma_\ca\mapsto\Gamma_W$ is a bijection between the $G$-gradings on $\ca$ and the $G$-gradings on $W$. 
It induces a bijection between the isomorphism classes of these gradings.
\item[2)] The correspondence $\Gamma_\ca\mapsto\Gamma_S$ is a bijection between the $S$-admissible $G$-gradings on $\ca$ and the $G$-gradings on $S^{(1)}$. 
It induces a bijection between the isomorphism classes of $G$-gradings on $S^{(1)}$ and the $\Aut_S(\ca)$-orbits of the $S$-admissible $G$-gradings on $\ca$.
\item[3)] The correspondence $\Gamma_\ca\mapsto\Gamma_H$ is a bijection between the $H$-admissible $G$-gradings on $\ca$ and the $G$-gradings on $H^{(2)}$. 
It induces a bijection between the isomorphism classes of $G$-gradings on $H^{(2)}$ and the $\Aut_H(\ca)$-orbits of the $H$-admissible $G$-gradings on $\ca$.
\end{enumerate}
\end{corollary}

\begin{proof}
Let $\Gamma:L=\bigoplus_{g\in G}L_g$ be a $G$-grading. Replacing $G$ with the subgroup generated by the support, 
we may assume that $G$ is finitely generated. We also obtain that the corresponding morphism $G^D\to\AAut(L)$ is a closed imbedding. 

Consider the case $L=W$. 
Due to the isomorphism $\Ad:\AAut(\ca)\to\AAut(L)$, we obtain a closed imbedding $G^D\to\AAut(\ca)$, which corresponds to a $G$-grading 
$\Gamma_\ca:\ca=\bigoplus_{g\in G}\ca_g$ (whose support also generates $G$, since otherwise we would not have a closed imbedding). 
The induced $G$-grading $\Gamma_W$ on $W$ is obtained by inducing the 
$\FF G$-comodule structure from $\ca$ to $\End(\ca)$ and then restricting to $L$, which agrees with how $\Ad:\AAut(\ca)\to\AAut(L)$ is defined. Therefore, $\Gamma=\Gamma_W$.

In the case $L=S^{(1)}$, we obtain a closed imbedding $G^D\to\AAut_S(\ca)$, so the subspace $\langle\oS\rangle$ of $\Omega^m$ is $G^D$-invariant, i.e., 
$\langle\oS\rangle$ is an $\FF G$-subcomodule. Hence $\oS$ is a homogeneous element in the corresponding $G$-grading $\Gamma_\ca:\ca=\bigoplus_{g\in G}\ca_g$. 

The proof in the case $L=H^{(2)}$ is similar. 
\end{proof}

\begin{remark}
It follows from the proof that the supports of the gradings $\Gamma_\ca$, $\Gamma_W$, $\Gamma_S$ and $\Gamma_H$ generate the same subgroup in $G$.
\end{remark}

We will now describe all possible $G$-gradings on $\ca=\ca(m;\ul{1})$.

\begin{prop}
\label{description_gradings_O}
Let $\ca=\ca(m;\ul{1})$ and let $\M$ be its unique maximal ideal. Let $G$ be an abelian group and let $\ca=\bigoplus_{g\in G}\ca_g$ be a $G$-grading.
\begin{enumerate}
\item[1)] There exist elements $y_1,\ld,y_m$ of $\M$ and $0\leq s\leq m$ such that 
the elements $1+y_1,\ld,1+y_s,y_{s+1},\ld,y_m$ are $G$-homogeneous and $\{y_1,\ld,y_m\}$ is a basis of $\M$ modulo $\M^2$. 
\item[2)] Let $\sd=\{g\in G\;|\;\ca_g\not\subset\M\}$. Then $\sd$ is an elementary $p$-subgroup of $G$.
\item[3)] Let $\{b_1,\ld,b_s\}$ be a basis of $\sd$. Then the elements $y_1,\ld,y_m$ can be chosen in such a way that the degree of $1+y_i$ is $b_i$, for all $i=1,\ld,s$.
\end{enumerate}
\end{prop}

\begin{proof}
1) Pick a basis for $\ca$ consisiting of $G$-homogeneous elements and select a subset $\{f_1,\ld,f_m\}$ of this basis  that is linearly independent modulo $\FF 1\oplus\M^2$. 
Order the elements $f_i$ so that $f_1,\ld,f_s$ have a nonzero constant term and $f_{s+1},\ld,f_m$ belong to $\M$. Rescale $f_1,\ld,f_s$ so that the constant term 
is $1$. Let $y_i=f_i-1$ for $i=1,\ld,s$ and $y_i=f_i$ for $i=s+1,\ld,m$. Then $y_1,\ld,y_m$ is a basis of $\M$ modulo $\M^2$.

2) Clearly, $e\in\sd$. If $a,b\in\sd$, then there exist elements $u\in\ca_a$ and $v\in\ca_b$ that are not in $\M$. Then the element $uv\in\ca_{ab}$ is not in $\M$, so 
$ab\in\sd$. Also, since $u^p$ is a nonzero scalar, we have $a^p=e$. It follows that $\sd$ is an elementary $p$-subgroup.

3) Any element of $\ca$ can be uniquely written as a (truncated) polynomial in the variables $1+y_1,\ld,1+y_s,y_{s+1},\ld,y_m$. Hence, for any $g\in G$,
\begin{equation}
\label{g_component}
\ca_g=\lspan{(1+y_1)^{j_1}\cdots(1+y_s)^{j_s}y_{s+1}^{j_{s+1}}\cdots y_m^{j_m}\;|\;0\leq j_i<p, a_1^{j_1}\cdots a_m^{j_m}=g},
\end{equation}
where $a_1,\ld,a_m\in G$ are the degrees of $1+y_1,\ld,1+y_s,y_{s+1},\ld,y_m$, respectively. It follows that $a_1,\ld,a_s$ generate $\sd$. Suppose they do not form a basis 
of $\sd$ --- say, $a_s=a_1^{\ell_1}\cdots a_{s-1}^{\ell_{s-1}}$. Set $\wt{y}_i=y_i$ for $i\neq s$ and 
\[
\wt{y}_s:=1+y_s-(1+y_1)^{\ell_1}\cdots(1+y_{s-1})^{\ell_{s-1}}.
\]
Then $1+\wt{y}_1,\ld,1+\wt{y}_{s-1},\wt{y}_s,\ld,\wt{y}_m$ are homogeneous of degrees $a_1,\ld,a_m$, respectively. Also, $\wt{y}_s\in\M$ and 
\[
\wt{y}_s=y_s-(\ell_1 y_1+\cdots+\ell_{s-1} y_{s-1})\pmod{\M^2},
\]
so $\wt{y}_1,\ld,\wt{y}_m$ still form a basis of $\M$ modulo $\M^2$. We have decreased $s$ by $1$. Repeating this process as necessary, we may assume that 
$\{a_1,\ld,a_s\}$ is a basis of $\sd$. Finally, if $\{b_1,\ld,b_s\}$ is another basis of $\sd$, we can write $b_j=\prod_{i=1}^{s}a_i^{\ell_{ij}}$, where $(\ell_{ij})$ is 
a non-degenerate matrix with entries in the field $GF_p$. Set   
\[
\wt{y}_j:=\prod_{i=1}^s(1+y_i)^{\ell_{ij}}-1\mbox{ for }j=1,\ld,s,
\]
and $\wt{y}_j=y_j$ for $j=s+1,\ld,m$. Then $\wt{y}_1,\ld,\wt{y}_m$ form a basis of $\M$ modulo $\M^2$, and $1+\wt{y}_j$ is homogeneous of degree $b_j$, $j=1,\ld,s$.
\end{proof}

\begin{remark}
Without loss of generality, assume that $G$ is generated by the support of the grading $\ca=\bigoplus_{g\in G}\ca_g$. 
Let $\Q$ be the image of $G^D$ under the corresponding closed imbedding $G^D\to\AAut(\ca)$. Let $\Hs=\mathrm{Stab}_{\AAut(\ca)}(\M)$. 
(In fact, $\Hs=\Aut(\ca)$, regarded as the largest smooth subgroupscheme of $\AAut(\ca)$.) Let $\Q_0=\Q\cap\Hs$. 
Then $\sd$ is the subgroup of $G$ corresponding to the Hopf ideal of $\FF G$ defining the subgroupscheme $\Q_0$ of $\Q$.
\end{remark}

\begin{proof}
Let $I_0$ be the Hopf ideal defining the subgroupscheme $\Q_0$ and let $G_0$ be the corresponding subgroup of $G$. 
Consider the coarsening $\ca=\bigoplus_{\bg\in G/G_0}\ca_\bg$ of the $G$-grading induced by the natural homomorphism $G\to G/G_0$, i.e., $\ca_\bg=\bigoplus_{g\in\bg}\ca_g$. 
This coarsening corresponds to the subgroupscheme $\Q_0\subset\Q$. Since $\Q_0$ stabilizes the subspace $\M\subset\ca$, we have $\M=\bigoplus_{\bg\in G/G_0}(\ca_\bg\cap\M)$. 
Hence $\ca_\bg\subset\M$ for $\bg\neq\be$ and $\ca_\be=\FF 1\oplus(\ca_\be\cap\M)$. Hence $\ca_g\subset\M$ for all $g\notin G_0$, which proves $\sd\subset G_0$. 
To prove that $\sd=G_0$, consider the Hopf ideal $I$ of $\FF G$ corresponding to $\sd$. Then $I\subset I_0$. 
The subgroupscheme $\wt{\Q}$ of $\Q$ defined by $I$ acts trivially on each $\ca_g$ with $g\in\sd$. It follows that $\wt{\Q}$ stabilizes $\M$. 
Hence $\wt{\Q}\subset\Q_0$ and $I\supset I_0$.
\end{proof}

The description of $G$-gradings on $\ca(m;\ul{1})$ resembles the description of $G$-gradings on the matrix algebra $M_n(\FF)$ --- see e.g. \cite{BK09,Eld09}. 
Namely, Proposition \ref{description_gradings_O} shows that the $G$-graded algebra $\ca(m;\ul{1})$ is isomorphic to the tensor product 
$\ca(s;\ul{1})\ot\ca(m-s;\ul{1})$ where the first factor has a {\em division grading} (in the sense that each homogeneous component is spanned by an invertible element) 
and the second factor has an {\em elementary grading} (in the sense that it is induced by a grading of the underlying vector space $\M/\M^2$). 
The isomorphism in question is, of course, the one defined by $y_1\mapsto x_1\ot 1,\ld,y_s\mapsto x_s\ot 1$ and $y_{s+1}\mapsto 1\ot x_1,\ld,y_m\mapsto 1\ot x_{m-s}$.
The first factor, $\ca(s;\ul{1})$, is isomorphic to the group algebra $\FF P$ as a $G$-graded algebra 
(where $\FF P$ has the standard $P$-grading, which is regarded as a $G$-grading).

To state the classification of $G$-gradings on $\ca$ up to isomorphism, we introduce some notation.

\begin{df}
\label{cf_gradings_O}
Let $\sd\subset G$ be an elementary $p$-subgroup of rank $s$, $0\le s\le m$. Let $t=m-s$ and let 
\[
\te=(g_1,\ld,g_t)\in G^t.
\]
We endow the algebra $\ca=\ca(m;\ul{1})$ with a $G$-grading as follows. 
Select a basis $\{b_1,\ld,b_s\}$ for $\sd$ and declare the degrees of $1+x_1,\ld,1+x_s$ to be $b_1,\ld,b_s$, respectively. 
Declare the degrees of $x_{s+1},\ld,x_m$ to be $g_1,\ld,g_t$, respectively.
We will denote the resulting $G$-grading on $\ca$ by $\Gamma_\ca(G,b_1,\ld,b_s,g_1,\ld,g_t)$. Since the gradings corresponding to different choices of basis for 
$\sd$ are isomorphic to each other, we will also denote this grading (abusing notation) by $\grca$.
\end{df}

\begin{df}
\label{equiv_datum}
Let $\te,\wt{\te}\in G^t$. We will write $\te\sim\wt{\te}$ if there exists a permutation $\pi$ of the set $\{1,\ld,t\}$ such that $\wt{g_i}\equiv g_{\pi(i)}\pmod{\sd}$ for all $i=1,\ld,t$.
\end{df}

\begin{theorem}
\label{classification_gradings_O}
Let $\FF$ be an algebraically closed field of characteristic $p>0$. Let $G$ be an abelian group. Let $\ca=\bigoplus_{g\in G}\ca_g$ be a grading on the algebra 
$\ca=\ca(m;\ul{1})$ over $\FF$. Then the grading is isomorphic to some $\grca$ as in Definition \ref{cf_gradings_O}. 
Two $G$-gradings, $\Gamma_\ca(G,\sd,\te)$ and $\Gamma_\ca(G,\wt{\sd},\wt{\te})$, are isomorphic if and only if 
$\sd=\wt{\sd}$ and $\te\sim\wt{\te}$ (Definition \ref{equiv_datum}).
\end{theorem}

\begin{proof}
Let $y_1,\ld,y_m$ be as in Proposition \ref{description_gradings_O}. Let $g_1,\ld,g_t\in G$ be the degrees of $y_{s+1},\ld,y_m$, respectively. 
Then the automorphism of $\ca$ defined by $y_i\mapsto x_i$, $i=1,\ld,m$, sends the grading $\ca=\bigoplus_{g\in G}\ca_g$ to $\Gamma_\ca(G,b_1,\ld,b_s,g_1,\ld,g_t)$.

If $\wt{g_i}=g_{\pi(i)}$, $i=1,\ld,t$, for some permutation $\pi$, then the automorphism of $\ca$ defined by
\begin{equation}
\label{auto_permutation}
x_i\mapsto x_i\mbox{ for }i=1,\ld,s\mbox{ and }x_{s+i}\mapsto x_{s+\pi(i)}\mbox{ for }i=1,\ld,t,
\end{equation}
sends $\Gamma_\ca(G,\sd,\wt{\te})$ to $\Gamma_\ca(G,\sd,\te)$.

If $\wt{g}_i=g_i b_1^{\ell_{i1}}\cdots b_s^{\ell_{is}}$, then the automorphism of $\ca$ defined by 
\begin{equation}
\label{auto_shift}
x_i\mapsto x_i\mbox{ for }i=1,\ld,s\mbox{ and }x_{s+i}\mapsto x_{s+i}\prod_{j=1}^s(1+x_j)^{\ell_{ij}}\mbox{ for }i=1,\ld,t, 
\end{equation}
sends $\Gamma_\ca(G,\sd,\wt{\te})$ to $\Gamma_\ca(G,\sd,\te)$.

Hence, if $\te\sim\wt{\te}$ as in Definition \ref{equiv_datum}, then $\Gamma_\ca(G,\sd,\wt{\te})$ is isomorphic to $\Gamma_\ca(G,\sd,\te)$.

It remains to show that the subgroup $\sd$ and the equivalence class of $\te$ are invariants of the $G$-graded algebra $\ca=\bigoplus_{g\in G}\ca_g$.
This is obvious for $\sd$, since $\sd=\{g\in G\;|\;\ca_g\not\subset\M\}$. Let $\bG=G/\sd$ and consider the coarsening of the $G$-grading, $\ca=\bigoplus_{\bg\in\bG}\ca_\bg$, 
induced by the natural homomorphism $G\to \bG$. It follows from the definition of $\sd$ that $\M$ is a $\bG$-graded subspace of $\ca$.
Consequently, $\M^2$ is also a $\bG$-graded subspace, and the quotient $V:=\M/\M^2$ inherits a $\bG$-grading: 
\[
V=V_{\wb{a}_1}\oplus\cdots\oplus V_{\wb{a}_\ell}.
\]
Let $k_i=\dim V_{\wb{a}_i}$ if $\wb{a}_i\neq\wb{e}$ and $k_i=\dim V_{\wb{a}_i}-s$ if $\wb{a}_i=\wb{e}$. 
Clearly, $\wb{a}_1,\ld,\wb{a}_\ell$ and $k_1,\ld,k_\ell$ are invariants of the $G$-graded algebra $\ca=\bigoplus_{g\in G}\ca_g$.
If the $G$-grading on $\ca$ is $\grca$, then, up to a permutation, $g_1 P=\ld=g_{k_1}P=\wb{a}_1$, $g_{k_1+1}P=\ld=g_{k_1+k_2}P=\wb{a}_2$, and so on.
\end{proof}

\begin{remark}
Instead of using $\te=(g_1,\ld,g_t)$ where some of the cosets $g_i P$ may be equal to each other, one can take multiplicities, 
\[
\ke=(k_1,\ld,k_\ell)\mbox{ where }k_i\mbox{ are positive integers},
\]
with $|\ke|:=k_1+\cdots+k_\ell=t$,  
\[
\te=(g_1,\ld,g_\ell)\mbox{ where }g_i\in G\mbox{ are such that }g_i^{-1}g_j\notin\sd\mbox{ for all }i\neq j,
\]
and write 
\[
\grcal=\Gamma_\ca(G,P,\underbrace{g_1,\ld,g_1}_{k_1\mbox{ times}},\ld,\underbrace{g_\ell,\ld,g_\ell}_{k_\ell\mbox{ times}}).
\]
Then $\Gamma_\ca(G,\sd,\ke,\te)$ is isomorphic to $\Gamma_\ca(G,\wt{\sd},\wt{\ke},\wt{\te})$ if and only if 
$\ke$ and $\wt{\ke}$ have the same number of components $\ell$ and there exists a permutation $\pi$ of the set 
$\{1,\ld,\ell\}$ such that $\wt{k}_i=k_{\pi(i)}$ and $\wt{g_i}\equiv g_{\pi(i)}\pmod{\sd}$ for all $i=1,\ld,\ell$. 
\end{remark}

\begin{df}
\label{cf_fine_gradings_O}
Fix $0\leq s\leq m$. For a multi-index $\alpha\in\ZZ^m$, let 
\[
\wb{\alpha}:=(\alpha_1+p\ZZ,\ld,\alpha_s+p\ZZ, \alpha_{s+1},\ld,\alpha_m)\in\ZZ_p^s\times\ZZ^{m-s}.
\] 
Define a $\ZZ_p^s\times\ZZ^{m-s}$-grading on $\ca=\ca(m;\ul{1})$ by declaring the degree of $1+x_i$, $i=1,\ld,s$, and the degree of $x_i$, $i=s+1,\ld,m$, 
to be $\wb{\veps_i}$.
This is the grading $\grca$ where $G=\ZZ_p^s\times\ZZ^{m-s}$ (written additively), $\sd=\ZZ_p^s$, and $\te=(\wb{\veps_{s+1}},\ld\wb{\veps_{m}})$.
We will denote this grading by $\Gamma_\ca(s)$.
\end{df}

\begin{corollary}
\label{fine_gradings_O}
Let $\ca=\ca(m;\ul{1})$. Then, up to equivalence, there are exactly $m+1$ fine gradings of $\ca$. They are $\Gamma_\ca(s)$, $s=0,\ld,m$. 
The universal group of $\Gamma_\ca(s)$ is $\ZZ_p^s\times\ZZ^{m-s}$.
\end{corollary}

\begin{proof}
All homogeneous components of $\Gamma_\ca(s)$ are 1-dimensional, so it is a fine grading. All relations in the grading group $\ZZ_p^s\times\ZZ^{m-s}$ come from 
the fact that $0\neq (\ca_g)^p\subset\ca_e$ for certain elements $g$. Hence $\ZZ_p^s\times\ZZ^{m-s}$ is the universal group of $\Gamma_\ca(s)$. 

For any abelian group $G$ and a $p$-subgroup $\sd\subset G$ with a basis $\{b_1,\ld,b_s\}$, any $G$-grading $\Gamma_\ca(G,b_1,\ld,b_s,g_1,\ld,g_{m-s})$ is induced from the 
$\ZZ_p^s\times\ZZ^{m-s}$-grading $\Gamma_\ca(s)$ by the homomorphism $\ZZ_p^s\times\ZZ^{m-s}\to G$ defined by 
\[
\wb{\veps_i}\mapsto b_i\mbox{ for }i=1,\ld,s,\mbox{ and }\wb{\veps_i}\mapsto g_{i-s}\mbox{ for }i=s+1,\ld,m.
\]
It follows that, up to equivalence, there are no other fine gradings.
The gradings $\Gamma_\ca(s)$ are pair-wise non-equivalent, because their universal groups are non-isomorphic.
\end{proof}

\begin{df}
\label{cf_gradings_W}
The $G$-grading induced by $\Gamma_\ca(G,b_1,\ld,b_s,g_1,\ld,g_t)$ (Definition \ref{cf_gradings_O}) on the Lie algebra $W$ will be denoted by 
$\Gamma_W(G,b_1,\ld,b_s,g_1,\ld,g_t)$ or $\grW$. Explicitly, we declare the degree of each element 
\[
(1+x_1)^{\alpha_1}\cdots(1+x_s)^{\alpha_s}x_{s+1}^{\alpha_{s+1}}\cdots x_m^{\alpha_m}\D_i\mbox{ where }\alpha\in\ZZ^{(m;\ul{1})}, 1\le i\le m,
\]
to be 
\[
b_1^{\alpha_1-\delta_{i,1}}\cdots b_s^{\alpha_s-\delta_{i,s}} g_1^{\alpha_{s+1}-\delta_{i,s+1}}\cdots g_t^{\alpha_m-\delta_{i,m}},
\]
where $\delta_{i,j}$ is the Kronecker delta. In particular, the gradings induced by $\Gamma_\ca(s)$ (Definition \ref{cf_fine_gradings_O}) will be denoted by $\Gamma_W(s)$.
\end{df}

The following is a generalization of a result in \cite{Dem70} (see also \cite[Corollary 7.5.2]{St}) on maximal tori of the restricted Lie algebra $W$, 
which corresponds to the case when $G$ is an elementary $p$-group.

\begin{theorem}
\label{classification_gradings_W}
Let $\FF$ be an algebraically closed field of characteristic $p>0$. Let $G$ be an abelian group. Let $W=W(m;\ul{1})$ over $\FF$. 
Assume $m\geq 3$ if $p=2$ and $m\geq 2$ if $p=3$. Then any grading $W=\bigoplus_{g\in G}W_g$ is isomorphic to some $\grW$ as in Definition \ref{cf_gradings_W}. 
Two $G$-gradings, $\Gamma_W(G,\sd,\te)$ and $\Gamma_W(G,\wt{\sd},\wt{\te})$, are isomorphic if and only if $\sd=\wt{\sd}$ and $\te\sim\wt{\te}$ 
(Definition \ref{equiv_datum}).
\end{theorem} 

\begin{proof}
A combination of Theorem \ref{classification_gradings_O} and Corollary \ref{gradings_come_from_O}.
\end{proof}

\begin{corollary}
\label{fine_gradings_W}
Let $W=W(m;\ul{1})$. Assume $m\geq 3$ if $p=2$ and $m\geq 2$ if $p=3$. Then, up to equivalence, there are exactly $m+1$ fine gradings of $W$. 
They are $\Gamma_W(s)$, $s=0,\ld,m$. The universal group of $\Gamma_W(s)$ is $\ZZ_p^s\times\ZZ^{m-s}$.$\hfill{\square}$
\end{corollary}

We now turn to special algebras.

\begin{prop}
\label{description_gradings_S}
In the notation of Proposition \ref{description_gradings_O}, assume that $\ca=\bigoplus_{g\in G}\ca_g$ is an $S$-admissible $G$-grading of degree $g_0$. 
Then the elements $y_1,\ld,y_m$ can be chosen in such a way that the degrees $a_1,\ldots,a_m\in G$ of $1+y_1,\ld,1+y_s,y_{s+1},\ld,y_m$ (respectively)
satisfy the equation $g_0=a_1\cdots a_m$.
\end{prop}

\begin{proof}
Choose elements $y_1,\ld,y_m$ as in Proposition \ref{description_gradings_O}. 
Let $a_1,\ld,a_m\in G$ be the degrees of the elements $1+y_1,\ld,1+y_s,y_{s+1},\ld,y_m$, respectively.
We are going to adjust $y_1,\ld,y_m$ to make $a_1,\ld,a_m$ satisfy the above equation. 
 
The form $dy_1\wedge\cdots\wedge dy_m$ is $G$-homogeneous of degree 
$a_0:=a_1\cdots a_m$. On the other hand, we have
\[
dy_1\wedge\cdots\wedge dy_m=f\oS\mbox{ where }f=\det(\D_j y_i).
\]
Since $\oS$ is $G$-homogeneous of degree $g_0$, we conclude that $f$ is $G$-homogeneous of degree $a_0 g_0^{-1}$. Since $f\notin\M$, we have $a_0 g_0^{-1}\in\sd$.

First consider the case $s=m$. Then $a_0\in\sd$ and thus $g_0\in\sd$. 
Also, the $G$-grading in this case is the eigenspace decomposition of $\ca$ with respect to a torus $T\subset\Der(\ca)=W$, 
where $T$ is isomorphic to the group of additive characters of $\sd$, so $T$ has rank $s=m$. 
If $g_0=e$, then $\oS$ is $T$-invariant, so $T\subset\mathrm{Stab}_W(\oS)=S$, which is a contradiction, because the toral rank of $S=\Sm$ is less than $m$ 
(in fact, it is $m-1$). Therefore, in this case we necessarily have $g_0\neq e$. 
It follows that there exists a basis $\{b_1,\ld,b_m\}$ of $\sd$ such that $g_0=b_1\cdots b_m$. By Proposition 
\ref{description_gradings_O}, we can replace $y_1,\ld,y_m$ with $\wt{y}_1,\ld,\wt{y}_m$ so that $1+\wt{y}_i$ is $G$-homogeneous of degree $b_i$, $i=1,\ld,m$.
The proof in this case is complete.

Now assume that $s<m$. Write $a_0 g_0^{-1}=a_1^{\ell_1}\cdots a_s^{\ell_s}$. Set $\wt{y}_i=y_i$ for $i<m$ and
\[
\wt{y}_m=y_m(1+y_1)^{-\ell_1}\cdots(1+y_s)^{-\ell_s}.
\]
Then $\wt{y}_m$ is $G$-homogeneous of degree $\wt{a}_m=a_m a_1^{-\ell_1}\cdots a_s^{-\ell_s}$ and hence $\wt{y}_1,\ld,\wt{y}_m$ are as desired.
\end{proof}

Recall that in Definition \ref{cf_gradings_O} of $\grca$, we had to choose a basis $\{b_1,\ld,b_s\}$ for $\sd$. The isomorphism class, i.e., the $\Aut(\ca)$-orbit, 
of the grading does not depend on this choice. Clearly, the grading is $S$-admissible of degree $g_0=b_1\cdots b_s g_1\ldots g_t$ 
and hence it induces a $G$-grading on the Lie algebra $S$ and its derived subalgebras. Let $L=\Sm^{(1)}$ if $m\ge 3$ and $L=\Sm^{(2)}$ if $m=2$. 
Since $g_0$ is $\Aut_S(\ca)$-invariant, the induced gradings on $L$ corresponding to different values of $g_0$ are not isomorphic. Conversely, suppose 
$\{\wt{b}_1,\ld,\wt{b}_s\}$ is another basis of $\sd$ such that $\wt{b}_1\cdots \wt{b}_s=b_1\cdots b_s$ (i.e., this basis leads to the same value of $g_0$). Write 
$\wt{b}_j=\prod_{i=1}^s b_i^{\alpha_{ij}}$ where $(\alpha_{ij})$ is a non-degenerate matrix with entries in the field $GF_p$. Set   
\begin{equation}
\label{auto_basis}
\wt{x}_j:=\prod_{i=1}^s(1+x_i)^{\alpha_{ij}}-1\mbox{ for }j=1,\ld,s,
\end{equation}
and $\wt{x}_j=x_j$ for $j=s+1,\ld,m$. Then $\wt{x}_1,\ld,\wt{x}_m$ form a basis of $\M$ modulo $\M^2$, and $1+\wt{x}_j$ is homogeneous of degree $\wt{b}_j$, $j=1,\ld,s$. 
One readily computes that
\begin{equation}
\label{J_auto_basis}
\det(\D_j\wt{x}_i)=\det(\alpha_{ij})\prod_{i=1}^s(1+x_i)^{-1+\sum_{j=1}^s \alpha_{ij}}.
\end{equation}
Now $\wt{b}_1\cdots \wt{b}_s=b_1\cdots b_s$ means that $\sum_{j=1}^s\alpha_{ij}=1$ for all $i$, so $\det(\D_j\wt{x_j})$ is in $\FF$.
Therefore, the automorphism of $\ca$ defined by $x_i\mapsto\wt{x}_i$, $i=1,\ld,m$, belongs to the subgroup $\Aut_S(\ca)$. 
We have proved that two $G$-gradings on $L$ arising from the same data $\sd$ and $\te$, but different choices of basis for $\sd$, are isomorphic if and only if 
they have the same value of $g_0$. This justifies the following:

\begin{df}
\label{cf_gradings_S}
Let $\sd$ and $\te$ be as in Definition \ref{cf_gradings_O}. Let $g_0\in G$ be such that 
\[
g_0g_1^{-1}\cdots g_t^{-1}\in\sd\setminus\{e\}.
\]
Select a basis $\{b_1,\ld,b_s\}$ for $\sd$ such that $g_0=b_1\cdots b_s g_1\cdots g_t$.  
The $G$-grading induced by $\Gamma_\ca(G,b_1,\ld,b_t,g_1,\ld,g_s)$ on the Lie algebra $S$ and its derived subalgebras will be denoted by 
$\Gamma_S(G,b_1,\ld,b_t,g_1,\ld,g_s)$ or $\grS$. In particular, the $\ZZ_p^s\times\ZZ^{m-s}$-grading induced by $\Gamma_\ca(s)$ (Definition \ref{cf_fine_gradings_O}, 
with $\{\wb{\veps_1},\ld,\wb{\veps_s}\}$ as a basis for $\ZZ_p^s$), will be denoted by $\Gamma_S(s)$.
\end{df}

The following is a generalization of a result in \cite{Dem70} (see also \cite[Theorem 7.5.5]{St}) on maximal tori of the restricted Lie algebra $CS$, 
which corresponds to the case when $G$ is an elementary $p$-group.

\begin{theorem}
\label{classification_gradings_S}
Let $\FF$ be an algebraically closed field of characteristic $p>3$. Let $G$ be an abelian group. Let $L=\Sm^{(1)}$ if $m\ge 3$ and $L=\Sm^{(2)}=\Hm^{(2)}$ if $m=2$ 
(a simple Lie algebra over $\FF$). Then any grading $L=\bigoplus_{g\in G}L_g$ is isomorphic to some $\grS$ as in Definition \ref{cf_gradings_S}. 
Two $G$-gradings, $\Gamma_S(G,\sd,\te, g_0)$ and $\Gamma_S(G,\wt{\sd},\wt{\te},\wt{g}_0)$, are isomorphic if and only if $\sd=\wt{\sd}$, 
$\te\sim\wt{\te}$ (Definition \ref{equiv_datum}) and $g_0=\wt{g}_0$.
\end{theorem}

\begin{proof}
First we show that the grading $L=\bigoplus_{g\in G}L_g$ is isomorphic to some grading $\grS$. We can apply Corollary \ref{gradings_come_from_O} to translate this 
problem to the algebra $\ca$. Let $\Gamma':\ca=\bigoplus_{g\in G}\ca'_g$ be the $S$-admissible grading on $\ca$, of some degree $g_0\in G$, that induces 
the grading $L=\bigoplus_{g\in G}L_g$. As usual, let $\sd=\{g\in G\;|\;\ca'_g\not\subset\M\}$ and let $s$ be the rank of $\sd$. 
By Proposition \ref{description_gradings_S}, there exist elements $y_1,\ld,y_m\in\M$ that form a basis of $\M$ mod $\M^2$ and such that $1+y_i$, $i\le s$, and 
$y_i$, $i>s$, are $G$-homogeneous of some degrees $a_i$, $i=1,\ld,m$, where $\{a_1,\ld,a_s\}$ is a basis of $P$ and $g_0=a_1\cdots a_m$. 
We want to show that there exists an automorphism in $\Aut_S(\ca)$ that sends $\Gamma'$ to the grading 
$\Gamma_\ca=\Gamma_\ca(G,a_1,\ld,a_s,a_{s+1},\ld,a_m)$. Denote the latter grading by $\ca=\bigoplus_{g\in G}\ca_g$. 
Let $\mu$ be the automorphism of $\ca$ defined by $y_i\mapsto x_i$, $i=1,\ld,m$. 
Then $\mu$ sends $\Gamma'$ to $\Gamma_\ca$, but $\mu$ may not belong to $\Aut_S(\ca)$. Write $\mu(\oS)=f\oS$ for some $f\in\ca$. 
Now $\mu(\oS)$ has degree $g_0$ relative to the grading induced on $\Omega^m$ by $\Gamma_\ca$, $\oS$ has degree $a_1\cdots a_m$ relative to the said grading, 
and $g_0=a_1\cdots a_m$, so we conclude that $f$ has degree $e$ relative to $\Gamma_\ca$. 
If $s=m$, this implies that $f$ is in $\FF$ and hence $\mu\in\Aut_S(\ca)$, completing the proof. So we assume $s<m$.

Now we follow the idea of the proof of \cite[Proposition 7.5.4]{St}, which is due to \cite{KuYa97}. Observe that
\begin{align*}
\mu(\oS)=&\mu\left(d(x_1 dx_2\wedge\cdots\wedge dx_m)\right)\\
=&d\left(\mu(x_1)d\mu(x_2)\wedge\cdots\wedge d\mu(x_m)\right)\\
=&d\left(\sum_{i=1}^m (-1)^{i-1}h_i dx_1\wedge\cdots\wedge dx_{i-1}\wedge dx_{i+1}\wedge\cdots\wedge dx_m\right)\\
=&\left(\sum_{i=1}^m \D_i h_i\right)\oS,
\end{align*}
where $h_1,\ld,h_m\in\ca$. Set $E:=\sum_{i=1}^m h_i\D_i\in W$. Since $\mu(\oS)=f\oS$, we have $\mathrm{div}(E)=\sum_{i=1}^m \D_i h_i=f$.
One can immediately verify that $\mathrm{div}(W_g)\subset\ca_g$ for all $g\in G$, where $\Gamma_W:W=\bigoplus_{g\in G}W_g$ 
is the grading induced on $W$ by $\Gamma_\ca$. (Also, this is a consequence of the fact that $\mathrm{div}:W\to\ca$ is $\AAut(\ca)$-equivariant.)
Since $f\in\ca_e$, replacing $E$ with its $G$-homogeneous component of degree $e$ will not affect the equation $\mathrm{div}(E)=f$, so we will assume that $E\in W_e$.

Define a $\ZZ$-grading on $\ca$ by declaring the degree of $x_1,\ld,x_s$ (or, equivalently, $1+x_1,\ld,1+x_s$) to be $0$ and the degree of $x_{s+1},\ld,x_m$ to be $1$.
This $\ZZ$-grading is compatible with the $G$-grading $\Gamma_\ca$ in the sense that the homogeneous components of one grading are graded subspaces of $\ca$ 
relative to the other grading. We will denote the filtration associated to this $\ZZ$-grading by $\ca_{\{\ell\}}$, $\ell=0,1,2,\ld$, to distinguish it from the filtration 
$\ca_{(\ell)}$ associated to the canonical $\ZZ$-grading.

Write $f=\sum_{k\ge 0}f_k$, where $f_k$ has degree $k$ in the $\ZZ$-grading and degree $e$ in the $G$-grading. 
Observe that the constant term of $f$ is equal to the constant term of $f_0$, so $f_0$ is an invertible element of $\ca$. 
Let $\tau_1$ be the automorphism of $\ca$ defined by 
\[
\tau_1(x_i)=x_i\mbox{ for }i<m\mbox{ and }\tau_1(x_m)=f_0^{-1}x_m. 
\]
Since $f_0$ has degree $e$ in the $G$-grading, $\tau_1$ preserves $\Gamma_\ca$, i.e., $\tau_1(\ca_g)=\ca_g$ for all $g\in G$. 
We also have $\tau_1(\ca_{\{\ell\}})=\ca_{\{\ell\}}$ for all $\ell$. 
Since $x_m$ has degree $1$ in the $\ZZ$-grading, it does not occur in $f_0$. Hence $\tau_1(f_0)=f_0$ and we can compute:
\begin{align*}
(\tau_1\circ\mu)(\oS)=&\tau_1(f\oS)=\tau_1(f)\tau_1(\oS)\\
=&(f_0+\tau_1(\wt{h}))f_0^{-1}\oS=(1+h)\oS,
\end{align*}
where $\wt{h}=\sum_{k\ge 1}f_k$ and $h=f_0^{-1}\tau_1(\wt{h})$. Note that $h\in\ca_{\{1\}}$. 

{\bf Claim}: For any $\ell=1,2,\ld$ there exists an automorphism $\tau_\ell$ of $\ca$ that preserves the $G$-grading $\Gamma_\ca$ and has the following property:
\begin{equation}
\label{induction_KuYa}
(\tau_\ell\circ\mu)(\oS)=(1+h)\oS\mbox{ where }h\in\ca_{\{\ell\}}.
\end{equation}

We proceed by induction on $\ell$. The basis for $\ell=1$ was proved above. Assume (\ref{induction_KuYa}) holds for some $\ell\ge 1$ and $\tau_\ell$.
Since $\tau_\ell$ preserves $\Gamma_\ca$, we have $1+h\in\ca_e$ and hence $h\in\ca_e$.
Write $h=\sum_{k\ge\ell}h_k$ where $h_k$ has degree $k$ in the $\ZZ$-grading and degree $e$ in the $G$-grading. 
As was shown above, there exists $E\in W_e$ such that  $\mathrm{div}(E)=1+h$. 
Write $E=\sum_{k\ge -1}E_k$ where $E_k$ has degree $k$ in the $\ZZ$-grading induced from our $\ZZ$-grading of $\ca$ and degree $e$ in the $G$-grading. 
Since $\mathrm{div}$ preserves the $\ZZ$-grading, we have $\mathrm{div}E_k=h_k$ for $k\ge 1$. Let $\wt{\tau}$ be the automorphism of $\ca$ defined by 
\[
\wt{\tau}(x_i)=x_i-E_\ell(x_i), i=1,\ld,m.
\]
Since $E_\ell\in W_e$, the automorphism $\wt{\tau}$ preserves the $G$-grading $\Gamma_\ca$. We also have $\wt{\tau}(f)=f\pmod{\ca_{\{k+1\}}}$ for all $f\in\ca_{\{k\}}$ 
and 
\[
\wt{\tau}(\oS)=(1-\mathrm{div}(E_\ell)+\wt{f})\oS=(1-h_\ell+\wt{f})\oS
\] 
for some $\wt{f}\in\ca_{\{2\ell\}}$. Hence we can compute: 
\begin{align*}
(\wt{\tau}\circ\tau_\ell\circ\mu)(\oS)=&\wt{\tau}((1+h)\oS)=\wt{\tau}(1+h)\wt{\tau}(\oS)\\
=&(1+h_\ell+\wh{f})(1-h_\ell+\wt{f})\oS=(1+\wt{h})\oS,
\end{align*}
where $\wh{f}\in\ca_{\{\ell+1\}}$ and $\wt{h}=-h_{\ell}^2+\wh{f}(1-h_\ell)+\wt{f}(1+h_\ell+\wh{f})\in\ca_{\{\ell+1\}}$.
Setting $\tau_{\ell+1}=\wt{\tau}\circ\tau_{\ell}$, we complete the induction step.

Set $\wt{\mu}=\tau_{\ell}\circ\mu$ for $\ell=(p-1)(m-s)+1$. Then $\wt{\mu}$ sends $\Gamma'$ to $\Gamma_\ca$ and belongs to $\Aut_S(\ca)$, since $\wt{\mu}(\oS)=\oS$.
We have proved the first assertion of the theorem.

Now, the subgroup $\sd$ and the equivalence class of $\te=(g_1,\ldots, g_t)$ are invariants of the $G$-graded algebra $\ca$, and $g_0$ is $\Aut_S(\ca)$-invariant. 
It remains to show that, if $\te\sim\wt{\te}$ and $b_1\cdots b_s g_1\cdots g_t=g_0=\tilde{b}_1\cdots \tilde{b}_s \tilde{g}_1\cdots \tilde{g}_t$ where $\{b_1,\ld,b_s\}$ and $\{\wt{b}_1,\ld,\wt{b}_s\}$ are bases of $P$ as in Definition \ref{cf_gradings_S}, then $\Gamma_\ca(G,b_1,\ldots,b_s,g_1,\ldots,g_t)$ and $\Gamma_\ca(G,\tilde{b}_1,\ldots,\tilde{b}_s,\tilde{g}_1,\ldots,\tilde{g}_t)$ are in the same $\Aut_S(\ca)$-orbit. 
Clearly, the automorphism (\ref{auto_permutation}) of $\ca$, determined by a permutation $\pi$ of $\{1,\ld,t\}$, belongs to $\Aut_S(\ca)$. 
So it suffices to consider the case $\wt{g}_i\equiv g_i\pmod{P}$. 
Write $\wt{b}_j=\prod_{i=1}^s b_i^{\alpha_{ij}}$ where $(\alpha_{ij})$ is a non-degenerate matrix with entries in the field $GF_p$.
Also write $\wt{g}_i=g_i\prod_{j=1}^s b_j^{\ell_{ij}}$, $i=1,\ld,t$. 
Then the composition $\mu$ of the automorphism defined by $x_j\mapsto\wt{x}_j$, $j\le s$, and $x_j\mapsto x_j$, $j>s$, 
where $\wt{x}_j$ are as in (\ref{auto_basis}), and the automorphism defined by (\ref{auto_shift}), sends 
$\Gamma_\ca(G,\wt{b}_1,\ld,\wt{b}_s,\wt{g}_1,\ld,\wt{g_t})$ to $\Gamma_\ca(G,b_1,\ld,b_s,g_1,\ld g_t)$.
Now, (\ref{J_auto_basis}) implies that 
\begin{align*}
\mu(\oS)=&\det(\alpha_{ij})\left(\prod_{i=1}^s(1+x_i)^{-1+\sum_{j=1}^s \alpha_{ij}}\right)\left(\prod_{i=1}^t\prod_{j=1}^s(1+x_j)^{\ell_{ij}}\right)\oS\\
=&\det(\alpha_{ij})\left(\prod_{i=1}^s(1+x_i)^{-1+\sum_{j=1}^s \alpha_{ij}+\sum_{j=1}^t \ell_{ji}}\right)\oS.
\end{align*}
On the other hand,
\[
\wt{b}_1\cdots\wt{b}_s \wt{g}_1\cdots\wt{g}_t=g_1\cdots g_t\prod_{i=1}^s b_i^{\sum_{j=1}^s \alpha_{ij}+\sum_{j=1}^t \ell_{ji}},
\]
so the equality $\wt{b}_1\cdots\wt{b}_s \wt{g}_1\cdots\wt{g}_t=g_0=b_1\cdots b_s g_1\cdots g_t$ implies that 
\[
\sum_{j=1}^s \alpha_{ij}+\sum_{j=1}^t \ell_{ji}=1\mbox{ for all }i
\] 
and hence $\mu\in\Aut_S(\ca)$.
\end{proof} 

\begin{corollary}
\label{fine_gradings_S}
Under the assumptions of Theorem \ref{classification_gradings_S}, there are, up to equivalence, exactly $m+1$ fine gradings of $L$. 
They are $\Gamma_S(s)$, $s=0,\ld,m$. The universal group of $\Gamma_S(s)$ is $\ZZ_p^s\times\ZZ^{m-s}$.$\hfill{\square}$
\end{corollary}


\end{document}